\documentclass[a4paper,11pt]{scrartcl}

\usepackage[a4paper,
inner=2.3cm, 
outer=2.3cm,
top=2.8cm,
bottom=2.8cm]{geometry}
\usepackage{indentfirst}
\usepackage{enumitem}

\usepackage{hyperref}
\hypersetup{
	colorlinks,
	linkcolor={red!50!black},
	citecolor={blue!50!black},
	urlcolor={blue!80!black}
}

\usepackage{amsmath}
\usepackage{amsfonts}
\usepackage{mathtools}
\usepackage{amssymb}
\usepackage{tikz-cd}
\usepackage{dsfont}
\usepackage{bm}
\usepackage{dsfont}
\usepackage{multirow}
\usepackage{rotating}

\newcounter{dummy} 
\numberwithin{dummy}{section}

\usepackage[hyperref,amsmath,thmmarks]{ntheorem}

\theoremstyle{plain}
\theoremseparator{.}
\theoremsymbol{\scalebox{0.7}{$ \square $}}
\theorembodyfont{}
\newtheorem{proposition}[dummy]{Proposition}

\theoremstyle{plain}
\theoremseparator{.}
\theoremsymbol{\scalebox{0.7}{$ \square $}}
\theorembodyfont{}
\newtheorem{lemma}[dummy]{Lemma}

\theoremstyle{plain}
\theoremseparator{.}
\theoremsymbol{\scalebox{0.7}{$ \square $}}
\theorembodyfont{}
\newtheorem{remark}[dummy]{Remark}

\theoremstyle{plain}
\theoremseparator{.}
\theoremsymbol{\scalebox{0.7}{$ \square $}}
\theorembodyfont{}

\theoremstyle{plain}
\theoremseparator{.}
\theoremsymbol{\scalebox{0.7}{$ \square $}}
\theorembodyfont{}
\newtheorem{example}[dummy]{Example}

\theoremstyle{plain}
\theoremseparator{.}
\theoremsymbol{\scalebox{0.7}{$ \square $}}
\theorembodyfont{}

\theoremstyle{plain}
\theoremseparator{.}
\theoremsymbol{\scalebox{0.7}{$ \square $}}
\theorembodyfont{}
\newtheorem{theorem}[dummy]{Theorem}

\theoremstyle{nonumberplain}
\theoremsymbol{$ \blacksquare $}
\theoremheaderfont{\scshape}
\newtheorem{proof}{Proof}

\allowdisplaybreaks

\usepackage{xcolor}
\definecolor{darkgreen}{rgb}{0.0,0.5,0.0}

\usepackage{tikz}
\usetikzlibrary{positioning, calc, shapes, fit, matrix, shapes.geometric}

\newcommand{\F}[1][R]{\mathbb{#1}} 
\newcommand{\setsep}{\ \middle|\ } 
\newcommand{\Eucprod}[2]{\left\langle #1,#2 \right\rangle} 
\newcommand{\Ltwonorm}[1]{\left\Vert #1 \right\Vert} 

\newcommand{\tr}{\mathop{\mathrm{tr}}}
\newcommand{\SE}{\mathrm{SE}}
\newcommand{\OG}{\mathrm{O}}
\newcommand{\SO}{\mathrm{SO}}

\newcommand{\sgn}{\mathop{\mathrm{sgn}}}

\newcommand{\dnb}{\bm{\epsilon}}
\newcommand{\qal}{\F[H]}
\newcommand{\dnal}{\F[D]}

\newcommand{\ii}{\mathbf{i}}
\newcommand{\jj}{\mathbf{j}}
\newcommand{\kk}{\mathbf{k}}
\newcommand{\dqal}{\F[DH]}
\newcommand{\dqmat}[1][n]{\dqal^{#1 \times #1}}
\newcommand{\dqm}[1][n]{\dqal^{#1}}

\usepackage[backend=bibtex, maxbibnames=5, style=numeric]{biblatex}

\usepackage{cleveref}
\usepackage{caption}
\usepackage{subcaption}

\title{$ \SE(3) $ Synchronization by Eigenvectors\\of Dual Quaternion Matrices}
\subtitle{}
\author{Ido Hadi, Tamir Bendory, Nir Sharon}
\date{}

\begin{document}
	
	\maketitle
	
	\begin{abstract}
		In synchronization problems, the goal is to estimate elements of a group from noisy measurements of their ratios. A popular estimation method for synchronization is the spectral method. It extracts the group elements from eigenvectors of a block matrix formed from the measurements. The eigenvectors must be projected, or ``rounded,'' onto the group. The rounding procedures are constructed ad hoc and increasingly so when applied to synchronization problems over non-compact groups.
		
		In this paper, we develop a spectral approach to synchronization over the non-compact group $\SE(3)$, the group of rigid motions of $\F^3$. 
		We based our method on embedding $\SE(3)$ into the algebra of dual quaternions, which has deep algebraic connections with the group $\SE(3)$. 
		These connections suggest a natural rounding procedure considerably more straightforward than the current state-of-the-art for spectral $\SE(3)$ synchronization, which uses a matrix embedding of $\SE(3)$. 
		We show by numerical experiments that our approach yields comparable results to the current state-of-the-art in $\SE(3)$ synchronization via the spectral method.
		Thus, our approach reaps the benefits of the dual quaternion embedding of $\SE(3)$, while yielding estimators of similar quality.
	\end{abstract}

	\section{Introduction}
	\label{sec:introduction}
	
	Synchronization problems arise as part of data processing pipelines in several contexts, including single-particle reconstruction in cryogenic electron microscopy \cite{Shkolnisky2012,Bendory2020}, structure from motion problems \cite{Oezyesil2017,Arrigoni2018} and simultaneous localization and mapping (SLAM) problems \cite{Arrigoni2016,Thrun2005}.
	A synchronization problem is an estimation problems over a group $G$ in which group elements $g_{1}, \dots, g_{n} \in G $ are estimated from measurements of their ratios $g_{i} g_{j}^{-1}$.
	These measurements are inherently ambiguous, since the set of ratios of $g_{1} g, \dots, g_{n} g \in G$ for any $g \in G$ are the same as the ratios of the original group elements.
	Thus, in synchronization problems, the goal is to estimate the group elements up to right-multiplication by an arbitrary element of $G$.
	
	The spectral method is a widely used method for solving synchronization problems. It seeks to approximate a solution by finding eigenvectors of a block matrix formed by embedding the group into a matrix algebra. The embedding enables the richer algebraic structure of matrix algebras, and in particular their spectral decomposition, to be leveraged to solve the synchronization problem.
	Importantly, the eigenvectors of the observation matrix are not themselves a solution of the synchronization problem. They have to be projected back onto the group, a projection that is typically non-linear. This procedure is often referred to as ``rounding'' in the synchronization literature. In typical rounding procedures, a matrix formed by these eigenvectors is chopped into blocks and these blocks are subsequently mapped onto elements of the group.
	
	Consider a synchronization problem over a relative simple compact group in the absence of measurement noise. 
	As we explain in greater detail in \Cref{sec:GroupSynchronizationSpectralMethod}, even in this ideal setting, 
	typically there is a gap between the synchronization problem and the eigenproblem used to solve it.
	Oversimplifying to an extent, we can say that every solution to synchronization problem is a solution of the eigenproblem used by the spectral method, but not every solution of the eigenproblem is a solution of the synchronization problem.
	This issue is exacerbated for synchronization over non-compact group, the focal point of this paper. 
	The solutions of the eigenproblem form a compact set, the set of unit eigenvectors which span the relevant eigenspaces of the measurement matrix. 
	Yet, unlike the case of synchronization over compact groups, the possible solutions of a synchronization problems over non-compact groups form a non-compact set. Therefore, the gap between the two problems can be said to be much larger.
	
	In practice, this gap between the eigenproblem and the synchronization problem must be bridged by the rounding procedure.
	It ensures that the output of the synchronization method is comprised of elements of the group.
	However, existing rounding procedures are made to do so in an ad hoc manner. Specifically, they do not stem from a well characterized relationship between the synchronization problem and the eigenproblem or the relationship between their respective algebraic contexts, the group and the algebra in which it is embedded.
	They are merely constructed to ensure that the result of a synchronization problem is an element of the group. 
	
	In this paper, we address the issues we raised above in the non-compact case of $\SE(3)$ synchronization, the group of rigid motions, rotations and translations, of $\F^{3}$. 
	In \cite{Arrigoni2016}, an analog of the spectral method was developed for $\SE(3)$ synchronization, which utilized an embedding of $\SE(3)$ into a matrix algebra. 
	We replace this matrix algebra embedding with an embedding into the algebra of dual quaternions.
	Unlike the matrix algebra embedding used by \cite{Arrigoni2016}, the algebra of dual quaternions is much closer to the structure of $\SE(3)$, in a sense we explain in \Cref{sec:SE3RepresentationOnDualQuaternions}.
	Two recent advances in dual quaternion linear algebra were a spectral theorem for matrices of dual quaternions \cite{Qi2021} and a power iteration capable of approximating its spectra \cite{Cui2023}.
	These two results allow us to develop an elegant spectral synchronization method. The algebraic context eliminates the gap between the synchronization problem and the eigenproblem, at least in the absence of measurement noise.
	In addition, the rounding step itself stems directly from the algebraic relationship between elements of the algebra itself and the set on which $\SE(3)$ is represented. In addition, it is computationally simpler than the rounding procedure of \cite{Arrigoni2016}.
	
	The structure of this paper is as follows. We begin in \Cref{sec:DualQuaternionAlgebraAndFreeModules} with a survey of the necessary background material. Among other things, we define the algebra of dual quaternions, matrix algebras with dual quaternion entries and their properties. 
	In \Cref{sec:SpectralApproachToSynchronization}, we delve deeper into synchronization problems and the spectral method. We define the problem rigorously, introduce the spectral method and conclude by laying out our approach, relying on dual quaternion embedding of $\SE(3)$.
	Finally, in \Cref{sec:NumericalExperiments}, we demonstrate empirically that our approach yields comparable results to the current state-of-the-art spectral method for $\SE(3)$ synchronization developed by \cite{Arrigoni2016}. Thus, we offer a better theoretical foundation tying $\SE(3)$ synchronization problems with eigenproblems, while still maintaining the quality of estimation of the current state-of-the-art spectral method.

	\section{Background: Dual Quaternions and Matrices of Dual Quaternions}
	\label{sec:DualQuaternionAlgebraAndFreeModules}
	
	We provide a succinct background information on dual quaternions and dual quaternion matrices.
	In \Cref{sec:DualNumbers} and \Cref{sec:DualQuaternions} we construct the dual quaternions and describe some of their properties.
	In \Cref{sec:SO3RepresentationOnQuaternions} we survey the well known representation of $\SO(3)$ on the algebra of quaternions and in \Cref{sec:SE3RepresentationOnDualQuaternions} we survey the lesser known representation of $\SE(3)$ on the algebra of dual quaternions.
	In \Cref{sec:DQSpectralTheorem} we define the free modules of dual quaternions and algebras of matrices with dual quaternion elements. This section culminates in the spectral theorem for matrices of dual quaternions, which was recently proved by \cite{Qi2021}.
	Finally, in \Cref{sec:PowerIteration} we survey the power iteration for matrices of dual quaternions, which was recently developed by \cite{Cui2023}.

	\subsection{Dual Numbers}
	\label{sec:DualNumbers}
	
	The algebra of dual numbers is the set $\dnal = \left\{ a + b \dnb \setsep a, b \in \F \right\}$ with addition and multiplication defined by
	\begin{align}
		\label{eq:DNAddition}
		\left(a_{1} + b_{1} \dnb\right) + \left(a_{2} + b_{2} \dnb \right)
		&= \left(a_{1} + a_{2}\right) + \left(b_{1} + b_{2} \right) \dnb \\
		\label{eq:DNProduct}
		\left( a_{1} + b_{1} \dnb \right) \left( a_{2} + b_{2} \dnb  \right)
		&= a_{1} a_{2} + \left( a_{1} b_{2} + b_{1} a_{2} \right) \dnb.
	\end{align}
	If $a + b \dnb \in \dnal$, it is convenient to refer to $a$ as the real coordinate and $b$ as the dual coordinate.
	It is easy to see from their definition that the dual numbers are defined analogously to the complex numbers. 
	Whereas the complex numbers are the algebra over the reals generated by $1 $ and $ i$ such that $i^{2} = - 1$, the algebra of dual numbers is generated by $1$ and $\dnb$ such that $\dnb^{2} = 0$.
	This similarity accounts for the similar properties of the algebras of dual numbers and complex numbers.
	Namely, the underlying vector space of both is $\F^{2}$, addition and multiplication operations of both are associative and commutative and both algebras have a unit, $1$.
	
	Despite the aforementioned similarities, these two algebras are fundamentally very different. The algebra of complex numbers is a field, while the algebra of dual numbers is not. 
	It has zero divisors, i.e., non-zero elements whose product is zero.
	Indeed, recall that an element $x$ of an algebra is said to be nilpotent if there is a positive integer $n$ such that $x^{n} = 0$.
	The only nilpotent complex number is $0$.
	The following proposition characterizes the nilpotent dual numbers, and it is obvious it has many non-zero nilpotents.
	\begin{proposition}
		\label{prop:DNNilpotents}
		An $x = a + b \dnb \in \dnal$ is nilpotent if and only if $a = 0$.
		Also, $x^{2} = 0$ for every nilpotent.
	\end{proposition}
	
	\begin{proof}
		If $a \ne 0$, it is easy to show that \eqref{eq:DNProduct} implies that $x^{n} = a^{n} + \left(\dots\right)\dnb \ne 0$.
		If $a = 0$, then $x^{2} = b^{2} \dnb^{2} = 0$.
	\end{proof}
	
	As the next proposition shows, nilpotent dual numbers are the only dual numbers which are not invertible.
	\begin{proposition}
		\label{prop:DNInvertible}
		If $x = a + b \dnb \in \dnal $, 
		$x$ is invertible if and only it is not nilpotent.
		In that case, $x^{-1} = a^{-1} - b a^{-2} \dnb$.
	\end{proposition}
	
	\begin{proof}
		Let $y = c + d \dnb$. From \eqref{eq:DNProduct} it follows that $x y = 1$ if 
		\begin{equation*}
			ac = 1, \qquad ad + bc = 0.
		\end{equation*}
		The first equation has a solution if and only if $a \ne 0$. By \Cref{prop:DNNilpotents}, this holds if and only if $x$ is not nilpotent. 
		The formula for $x^{-1}$ is easily obtained by solving this system of equations when $a \ne 0$.
	\end{proof}
	
	A final point of analogy between the dual and complex numbers, is that both admit a square root.
	Following \cite{Qi2022}, given $x = a + b \dnb$ and $y = c + d \dnb$, we write $x > y$ if $a > c$, or $a = c$ and $b > d$.
	We say a dual number $x = a + b \dnb \in \dnal$ is \textbf{non-negative} if $x \ge 0 $ and \textbf{positive} if $x > 0$.
	In \cite[Sec. 2]{Qi2022}, the following definition of a square root was given:
	\begin{equation}
		\label{eq:DNAbsoluteValue}
		\sqrt{x}
		= \begin{cases}
			\sqrt{a} + \frac{b}{2 \sqrt{a}} \dnb  & x \mbox{ is positive and invertible} , \\
			0 	& x = 0.
		\end{cases}
	\end{equation}
	In all other cases, the square root is undefined.
	Using this definition, it is possible to show that $\sqrt{x y } = \sqrt{x} \sqrt{y}$.
	This multiplicative identity implies that for any $x = a + b \dnb\in \dnal$ with $a \ne 0$, we have
	$\sqrt{x^{2}} = \left|a\right| + \sgn(a) b \dnb $, where $\sgn(a) = 1$ for $a>0$, zero when $a = 0$ and $-1$ when $a<0$.
	Such considerations motivated \cite{Qi2022} to define the following absolute value on dual numbers:
	\begin{equation*}
		\left|x\right|
		= \begin{cases}
			\left|a\right| + \sgn (a) b \dnb & 	a \ne 0 \\
			\left|b\right| \dnb & a = 0.
		\end{cases}
	\end{equation*}
	It has the following properties \cite[Theorem 2]{Qi2022}, which are generalized forms of the usual absolute value defined on the real line:
	\begin{theorem}
		For any $x, y \in \dnal$:
		\begin{enumerate}
			\item $\left|x\right| = 0$  if and only if $x = 0$.
			
			\item $\left|x\right| = x$ if $x \ge 0$ and $\left|x\right| > x$ otherwise.
			
			\item $\left|x\right| = \sqrt{x^{2}}$ if $x$ is invertible.
			
			\item $\left|x y\right| = \left|x\right| \left|y\right|$ for any $x, y \in \dnal$.
			
			\item $\left| x + y\right| \le \left|x\right| + \left|y\right|$.
			
		\end{enumerate}
	\end{theorem}

	\subsection{Dual Quaternions}
	\label{sec:DualQuaternions}
	
	The algebra of dual quaternions is easiest to define using the algebra of quaternions.
	The algebra of quaternions is the set $\qal = \left\{ a + b \ii + c \jj + d \kk \setsep a, b, c, d \in \F \right\}$.
	We typically refer to $b$, $c$ and $d$ as the $\ii$, $\jj$ or $\kk$ coordinates, respectively.
	In the literature, $a$ is sometimes referred to as the real part of the quaternion, but in order to eliminate confusion with the real and dual parts of a dual number, we simply refer to it as the first coordinate of the quaternion.
	Addition in $\qal$ is defined coordinate-wise and multiplication is defined by the following relations among its generators:
	\begin{equation}
		\label{eq:QuaternionGeneratorsRelations}
		\ii^{2} 
		= \jj^{2} 
		= \kk^{2} 
		= \ii \jj \kk 
		= -1.
	\end{equation}
	As a vector space, $\qal \cong \F^{4}$. Also, $\F \subset \qal$ by constraining the $\ii$, $\jj$ and $\kk$ coordinates to zero.
	The center of algebra, the set of elements for which multiplication commutes with all other elements of the algebra, is $\F$.
	There is an involution on the quaternions, defined analogously to the complex conjugation, $q^{*} = a - b \ii - c \jj - d \kk$ for $q = a + b \ii + c \jj + d \kk$.
	This involution is referred to as \textbf{quaternion conjugation} and we say that $q^{*}$ is the \textbf{conjugate} of $q$.
	Similarly to the complex norm, there is a norm on the quaternions defined as 
	$\left|q\right| \coloneqq \sqrt{q^{*} q} = \sqrt{a^{2} + b^{2} + c^{2} + d^{2}}$, 
	which can be shown to imply $\qal \cong \F^{4}$ as a normed vector space.
	We note here that $\left|q\right|^{2} = q^{*} q = q q^{*}$, and so the inverse of every non-zero $q \in \qal$ is $\frac{q^{*}}{q^{*} q}$.
	Overall, the algebra of quaternions is an involutive normed division algebra.
	
	The algebra of dual quaternions is $\dqal = \left\{ a + b \dnb \setsep a, b \in \qal \right\}$.
	Essentially, it means a dual quaternion is a dual number with quaternion coordinates instead of real coordinates.
	From this, we immediately see that $\dnal \subset \dqal$ and also that $\qal \subset \dqal$, the former by constraining $a, b \in \F$ and the latter by setting $b = 0$.
	Addition and multiplication are still defined by \eqref{eq:DNAddition} and \eqref{eq:DNProduct} with the quaternion binary operations replacing the real ones.
	The trivial involution on $\dnal$, the identity, and quaternion conjugation, induce an involution on $\dqal$, $x^{*} = a^{*} + b^{*} \dnb$ for $x = a + b \dnb \in \dqal$.
	Because~$\dnal$ is the center of $\dqal$, we can use the usual notation for division whenever $d = a + b \dnb \in \dnal$ is invertible and $x \in \dqal$, and so we write 
	\begin{equation}
		\label{eq:DNDivision}
		\frac{x}{d} 
		\coloneqq x d^{-1}
		= d^{-1} x.
	\end{equation}
	As a vector space, $\dqal \cong \F^{8}$.
	We use the same names for the coordinates we used for the dual numbers. Given $a + b \dnb \in \dqal$, we call $a$ the real coordinate and $b$ the dual coordinate.
	
	The algebra of dual numbers is a unital algebra, since both $\dnal$ and $\qal$ are.
	It is non-commutative, since $\qal$ is. 
	It is not a division algebra, since $\dnal$ is not.
	This combination of features set it apart from both the algebra of quaternions and the algebra of dual numbers.
	Despite this, many of their features generalize well to the dual quaternions.
	As we did for the dual numbers, we begin by noting that its nilpotents are characterized in exactly the same way as the nilpotents of the dual numbers.
	Indeed, one need only replace $\dnal$ with $\dqal$ in \Cref{prop:DNNilpotents} to obtain the proper characterization.
	
	Note here that zero is the only nilpotent in the algebra of quaternions.
	Therefore, one can say the first part of \Cref{prop:DNInvertible} holds for quaternions, namely, that a quaternion is invertible if and only if it is not nilpotent.
	We show now that the dual quaternions share this property.
	Before that, we prove a lemma pointing out several important features of the involutive structure of $\dqal$.
	\begin{lemma}
		\label{lem:DQInvolutiveStructure}
		If $x = a + b \dnb \in \dqal$, then
		$x^{*} x 
		= \left|a\right|^{2} 
		+ \left( a b^{*} + b a^{*} \right) \dnb $ is a dual number
		and $x^{*} x = x x^{*}$.
	\end{lemma}
	\begin{proof}
		The first equality is easily proved by substituting $x^{*}$ and $x$ into \eqref{eq:DNProduct}.
		To prove the result is a dual number, we note first that the norm of a quaternion is a real number. In addition, $q^{*} + q \in \F$ for every quaternion $q$ and $\left( a b^{*} \right)^{*} = b a^{*}$, which implies $a b^{*} + b a^{*}$ is also real.
		The second identity follows from a direct calculation.
	\end{proof}
	
	\begin{proposition}
		\label{prop:DQInvertible}
		If $x = a + b \dnb \in \dqal$, $x$ is invertible if and only if it is not nilpotent. In that case, $x^{-1} = \frac{x^{*}}{x^{*} x}$.
	\end{proposition}
	\begin{proof}
		If $a = 0$, we have $x y = b c \dnb \ne 1 $ for every $y = c + d \dnb \in \dqal$.
		Therefore, $x$ is invertible only if $a \ne 0$, that is, only if $x$ is not nilpotent.
		If $a \ne 0$, it follows from \Cref{lem:DQInvolutiveStructure} that the real coordinate of $x^{*} x$ is positive.
		Thus, by \Cref{prop:DNInvertible} $x^{*} x$ is invertible.
		Finally, taking $y = \frac{x^{*}}{x^{*} x}$ one obtains that $y x = 1$.
		The equality $x y = 1$ follows easily from the second equality in \Cref{lem:DQInvolutiveStructure}.
	\end{proof}
	
	In \cite{Qi2022}, absolute value was also defined for dual quaternions. If $x = a + b \dnb \in \dqal$, then
	\begin{equation}
		\label{eq:DQAbsoluteValue}
		\left| x \right|
		= \begin{cases}
			\left|a\right| + \frac{a b^{*} + b a^{*} }{2 \left|a\right|} \dnb		& 	a \ne 0 \\
			\left|b\right| \dnb		&	a = 0.
		\end{cases}
	\end{equation}
	It can be readily seen that when $x$ is a dual number, \eqref{eq:DQAbsoluteValue} and \eqref{eq:DNAbsoluteValue} agree.
	In \cite[Thm. 4]{Qi2022}, the following properties of the absolute value of dual quaternions were proved:
	\begin{theorem}
		\label{thm:DQAbsoluteValueProperties}
		Let $x, y \in \dqal$.
		\begin{enumerate}
			\item $\left|x\right|$ is a dual number.
			
			\item If $x$ is invertible, then $\left|x\right| = \sqrt{q q^{*}}$.
			
			\item $\left|x\right| = \left|x^{*}\right|$.
			
			\item $\left|x\right| \ge 0$ for all $x$ and $\left|x\right| = 0$ if and only if $x = 0$.
			
			\item $\left|x y\right| = \left|x\right| \left|y\right|$.
			
			\item $\left|x + y \right| \le \left|x\right| + \left|y\right|$.
			
		\end{enumerate}
	\end{theorem}

	\subsection{Representing $\SO(3)$ on the quaternions}
	\label{sec:SO3RepresentationOnQuaternions}
	
	\newcommand{\tal}{\F[T]}
	
	The algebra of quaternions, which we surveyed in \Cref{sec:DualQuaternions}, has a close relationship with the special orthogonal group $\SO(3)$, the group of norm and orientation-preserving linear transformations of the Euclidean space $\F^{3}$.
	The subset $\qal_{1} \coloneqq \left\{ q \in \qal \setsep \left| q \right| = 1 \right\}$ is actually a subgroup of the multiplicative group of $\qal$. This follows easily from the fact $\left|q\right|^{2} = q^{*} q = q q^{*}$. In particular, this also implies that the inverse of every $q \in \qal_{1}$ is its conjugate $q^{*}$, making the group structure of $\qal_{1}$ closed under involution and dependent on the involutive structure of the algebra of quaternions.
	A well known result states that $\qal_{1}$ is a double cover of $\SO(3)$ \cite[Sec. 4.9]{Garling2011}. 
	We succinctly construct the covering map here. 
	We begin by identifying $\F^{3} \cong \mathbb{T} $ as vector spaces, where $ \tal = \left\{ a \ii + b \jj + c \kk \setsep a, b, c \in \F\right\} \subset \qal$.
	Then, for every $q \in \qal_{1}$, define a map $\F^{3} \to \F^{3}$ by
	$\varphi(q) (x) = q x q^{*}$, where $x \in \F^{3}$ is treated as a quaternion in $\tal$ via the identification above.
	We then have the following:
	\begin{proposition}
		\label{prop:QuaternionDoubleCoverSO3}
		The covering map is a $2$-to-$1$ surjective group homomorphism $\varphi : \qal_{1} \to \SO(3)$, such that $\varphi^{-1} (g) = \left\{ q, -q \right\}$ for every $g \in \SO(3)$.
	\end{proposition}
	This double cover allows one to represent elements of $\SO(3)$ on $\qal_{1}$.
	More specifically, one can represent $\SO(3)$ on $\qal_{1+} = \left\{ q \in \qal_{1} \setsep q + q^{*} \ge 0 \right\}$ by choosing an appropriate $q \in \varphi^{-1} (g)$.

	\subsection{Representing $\SE(3)$ on the dual quaternions}
	\label{sec:SE3RepresentationOnDualQuaternions}
	
	\newcommand{\dtal}{\F[DT]}
	
	The \textbf{special Euclidean group} $\SE(3)$ is the group of rigid motions of $\F^{3}$, i.e., rotations and translations of $\F^{3}$, but not reflections.
	It is a semidirect product $\SE(3) = \SO(3) \ltimes \F^{3} $. 
	Its elements are most directly represented as pairs of the form $\left(R, t \right) \in \SO(3) \times \F^{3}$ with the group operation defined as $\left(R_{1}, t_{1}\right) \circ \left(R_{2}, t_{2}\right) = \left(R_{1} R_{2}, t_{1} + R_{1} t_{2} \right)$. Here, $\circ$ can be thought of as the composition of two affine transformations.

	Let $\dqal_{1} = \left\{ x \in \dqal \setsep x^{*} x = 1 \right\}$ be the set of \textbf{unit dual quaternions}.
	Let $\dtal = \left\{ 1 + \frac{1}{2} t \dnb \setsep t \in \tal \right\}$, where we defined $\tal$ in \Cref{sec:SO3RepresentationOnQuaternions}.
	Note that $\qal_{1} \subset \dqal_{1}$ and $\dtal \subset \dqal_{1}$.
	We state an analogue of \Cref{prop:QuaternionDoubleCoverSO3}. It is a synopsis of \cite[Chp. 9]{Selig2005} and its proof can be found there.
	\begin{proposition}
		\label{prop:DualQuaternionDoubleCoverSE3}
		We have:
		\begin{enumerate}
			\item $\dqal_{1} = \dtal \rtimes \qal_{1}$.
			
			\item Let $\psi: \dqal_{1} \to \SE(3)$ be defined by $\psi(q, t) (x) = \left(\varphi(q), t \right)$. Here, $\varphi$ is the double cover defined in \Cref{prop:QuaternionDoubleCoverSO3}, $q \in \qal_{1}$ and $1 + \frac{1}{2} t \dnb \in \dtal $ so $t$ is identified with an element of $\F^{3}$.
			Then $\psi$ is a  $2$-to-$1$ surjective group homomorphism such that $\psi^{-1} (g) = \left\{ - x, x\right\}$ for all $g \in \SE(3)$.
			
		\end{enumerate}
	\end{proposition}
	
	\Cref{prop:DualQuaternionDoubleCoverSE3} offers a way to represent elements of $\SE(3)$ in much the same way we represent elements of $\SO(3)$ on the quaternions.
	We represent $(R, \mathbf{t}) \in \SE(3)$ by $x = q t \in \dqal_{1}$,  where $q$ is a unit quaternion with non-negative first coordinate representing the rotation $R$ and $\mathbf{t}$ is represented by $t = 1 + \frac{1}{2} t' \dnb$, $t' = t_{1} \ii + t_{2} \jj + t_{3} \kk $.
	Its action on $\F^{3}$ can be defined similarly to what we saw in the quaternion case.
	Given $\mathbf{s} \in \F^{3}$, let $s = 1 + s' \dnb$, $s'=t_{1} \ii + t_{2} \jj + t_{3} \kk $. Note the absence of a $\frac{1}{2}$ factor in this embedding of $\F^{3}$ into the dual quaternions, compared to how the translational part is represented on $\dtal$ in \Cref{prop:DualQuaternionDoubleCoverSE3}.
	The action of $(R, \mathbf{t}) \in \SE(3) $ on $\mathbf{s}$ is then $s \mapsto x s x^{*}$. Indeed, we have:
	\begin{align*}
		x s x^{*}
		&= \left(q + \frac{1}{2} t' q \dnb \right)
		\left( 1 + s' \dnb \right)
		\left(q^{*} + \frac{1}{2} q^{*} t'^{*} \dnb \right) \\
		&= q q^{*}
		+\left( 
		\frac{1}{2} t' q q^{*}
		+ q s' q^{*}
		+ \frac{1}{2} q q^{*} t'^{*}
		\right) \dnb \\
		&= 1
		+ \left(
		q s' q^{*}
		+ t'
		\right) \dnb.
	\end{align*}
	In expanding the products in the second transition we used the fact $\dnb^{2} = 0$, which implied no products of two dual coordinates could appear in the result.
	The dual coordinate of the result is exactly the action of $(R,\mathbf{t})$ on $\mathbf{s}$, because $q s' q^{*}$ is the action of $R$ on $\mathbf{s}$ by \Cref{prop:QuaternionDoubleCoverSO3} and because of the manner in which we represented $\F^{3}$ on quaternions when we constructed $t'$ and $s'$.
	
	\textcite{Cui2023} proved that it is possible to project an almost arbitrary dual quaternion onto the unit dual.
	Their result is stated here as follows:
	\begin{theorem}
		\label{thm:ProjectionSE3}
			Let $ x = a + b \dnb \in \dqal $ be non-zero. 
			\begin{enumerate}
				\item If $a \ne 0$, then $\frac{q}{\left|q\right|}$ is a unit dual quaternion solving 
				\begin{equation}
					\label{eq:ProjectionOptProblem}
					\min_{v \in \dqal_{1}} \left| v - x \right|^{2}.
				\end{equation}
				
				\item If $a = 0$, then any $q' = a' + b' \dnb$ such that $a' = \frac{b}{\left|b\right|}$ and $b^{*} b' + b'^{*} b = 0 $ solves \eqref{eq:ProjectionOptProblem}.
				
			\end{enumerate}
	\end{theorem}
	This theorem exposes an aspect of the close tie between $\SE(3)$ and the algebra of dual quaternions. 
	We focus on its first part.
	It can be interpreted as the idea that every invertible dual quaternion is essentially a unit dual quaternion multiplied by a dual number. Thus, in a sense, the multiplicative structure of the algebra of dual quaternions is almost entirely generated by the group of unit dual quaternions.
	This property has a geometric interpretation via the absolute value function. Namely, dividing a dual quaternion by its absolute value yields a solution to the optimization problem~\eqref{eq:ProjectionOptProblem}, a sort of projection function. This function itself is defined via the dual quaternion absolute value~\eqref{eq:DQAbsoluteValue}, which itself intimately related to the algebraic structures on the algebra of dual quaternions as indicated in~\Cref{thm:DQAbsoluteValueProperties}.
	
	Finally, we refer the interested reader to \Cref{sec:RepresentingSE3Comparison}, where we compare various representations of $\SE(3)$, focusing on features which are perhaps most salient in applied work.

	\subsection{Matrices of Dual Quaternions and Their Spectral Decomposition}
	\label{sec:DQSpectralTheorem}
	
	Modules are an extensively studied generalization of the notion of a vector space to vectors and matrices with entries from an arbitrary algebra. 
	See \cite{Hazewinkel2004} for an introduction to module theory. Here, we provide a synopsis of the necessary information.
	A (right-)module $\F[M]$ over an algebra $\F[A]$ is a commutative group $\left( \F[M], + \right)$ along with a map $\F[M] \times \F[A] \to \F[M]$, $ (m, x) \mapsto m \cdot x $, which satisfies the following properties.
	First, it is distributive in the sense that $m \cdot (x + y) = m \cdot x + m \cdot y$ and $(m_{1} + m_{2}) \cdot x = m_{1} \cdot x + m_{2} \cdot x$.
	Second, the operation is associative, $\left(m \cdot x\right) \cdot y = m \cdot (x y)$.
	We often refer to this map as the action of $\F[A]$ on $\F[M]$ and often omit the multiplication sign.
	Oversimplifying to an extent, a module can be thought of as a vector space in which the underlying field was replaced by an arbitrary algebra.
	
	We focus on two types of free modules of the algebra of dual quaternions, $\dqal^{n}$ and $\dqmat$. Given a positive integer $n$, the \textbf{free $\dqal$-module} is $\dqal^{n} = \left\{ \mathbf{m} = \left(m_{1}, \dots, m_{n} \right)^{\top} \setsep m_{j} \in \dqal, j=1,\dots,n\right\}$. It is acted upon by $\dqal$ via entry-wise right-multiplication, i.e., $\left(\mathbf{m} x\right)_{j} = m_{j} x$.
	Addition is defined entry-wise.
	Moreover, $\dqmat$ can be thought of as the module of all $n \times n$ matrices of dual quaternions.
	As an $\dqal$-module, one can identify it with $\dqal^{n^{2}}$, and so addition is also performed entry-wise and so is right-multiplication by a dual quaternion.
	However, it actually is an algebra, not just a module, with the multiplication of matrices defined in the familiar way.
	Similarly, one can define the product $\mathbf{A} \mathbf{x}$ of a matrix $\mathbf{A} \in \dqmat$ and $\mathbf{x} \in \dqal^{n}$ in the usual way.
	It is instructive to think of $\mathbf{A}$ as a $\dqal$-linear operator on $\dqal^{n}$ in the sense that $\mathbf{A} \left( \mathbf{x} a + \mathbf{y} b \right) = \left(\mathbf{A} \mathbf{x} \right) a + \left(\mathbf{A} \mathbf{y} \right) b$  for any $\mathbf{x}, \mathbf{y} \in \dqal^{n}$ and $a, b\in \dqal$.
	
	Several types of real and complex matrices have dual quaternions counterparts. The conjugate transpose of a matrix $\mathbf{A} \in \dqmat$ is $\left(\mathbf{A}^{*}\right)_{i j} = \mathbf{A}_{j i}^{*} $.
	This can easily be shown to be an involution on $\dqmat$. 
	We say a matrix $\mathbf{A} \in \dqmat$ is Hermitian if $\mathbf{A} = \mathbf{A}^{*}$.
	It is unitary if $\mathbf{A}^{*} \mathbf{A} = \mathbf{A} \mathbf{A}^{*} = \mathbf{I}_{n}$, where $\mathbf{I}_{n}$ is the identity matrix, which is defined in usual way.
	
	Deepening the analogy with real and complex matrices, a notion of eigenvector and eigenvalue exists for matrices of dual quaternions.
	We say $\lambda \in \dqal$ is a (right-)eigenvalue of the matrix $\mathbf{A} \in \dqmat$ if there is $\mathbf{x} \in \dqal^{n}$ such that $\mathbf{A} \mathbf{x} = \mathbf{x} \lambda$.
	As the name suggests, the non-commutativity of $\dqal$ means one can also define a left-eigenvalue of a matrix of dual quaternions, though we will not be presenting any results on these. Whenever we speak of eigenvalues and eigenvectors of dual quaternion matrices we will be referring to right-eigenvalues and right-eigenvectors.
	
	We are now ready to state the spectral theorem for Hermitian matrices of dual quaternions \cite[Thm. 4.1]{Qi2021}:
	\begin{theorem}
		\label{thm:SpectralTheorem}
		Let $ \mathbf{A} \in \dqmat  $.
		If $ \mathbf{A} $ is Hermitian, then there are a unitary $ \mathbf{U} \in \dqmat $ and a diagonal $ \bm{\Sigma} \in \dqmat $ with dual number elements such that 
		$ \mathbf{A} = \mathbf{U} \bm{\Sigma} \mathbf{U}^{*} $.
		The diagonal elements of $\bm{\Sigma}$ have the form $\sigma_{j} = \alpha_{j} + \beta_{j} \dnb \in \dnal$ with $\alpha_{1} \ge \alpha_{2} \ge \dots \ge \alpha_{n}$. Here, $\sigma_{j}$ is the eigenvalue of $\mathbf{A}$ corresponding to the eigenvector $\mathbf{u}_{j}$, the $j$th column of $\mathbf{U}$.
	\end{theorem}
	This theorem essentially states that Hermitian matrices of dual quaternions have a full set of eigenvectors. Furthermore, in the dual quaternion case, as in the complex case, the eigenvalues turn out to be self-adjoint, that is, they satisfy $\sigma^{*} = \sigma$, with respect to the appropriate involution.
	This highlights the striking similarities between this theorem and the well-known spectral theorem of complex Hermitian matrices.

	\subsection{Power Iteration for Hermitian Matrices of Dual Quaternion}
	\label{sec:PowerIteration}
	
	In a recent paper, \textcite{Cui2023} developed a dual quaternion analog of the power iteration, capable of approximating the top eigenvalue of a Hermitian dual quaternion matrix and its corresponding eigenvalue. 
	We clarify below the sense in which an eigenvalue is the top eigenvalue.
	Let $\mathbf{A} \in \dqmat$  be a Hermitian matrix of dual quaternions.
	Given some initializer $\mathbf{v}^{{0}} \in \dqm$, the power iteration defined in \cite{Cui2023} is
	\begin{equation}
		\label{eq:PowerIteration}
		\mathbf{y}^{(k)}
		= \mathbf{A} \mathbf{v}^{(k-1)}, \quad
		\lambda^{(k-1)}
		= \left(  \mathbf{v}^{(k-1)} \right)^{*} \mathbf{y}^{(k)}, \quad
		\mathbf{v}^{(k)}
		= \frac{\mathbf{y}^{(k)}}{\Ltwonorm{\mathbf{y}^{(k)}}_{2}}.
	\end{equation}
	Here, the dual quaternion norm used is defined as follows:
	\begin{equation}
		\label{eq:DualQuaternionNorm}
		\Ltwonorm{\mathbf{x}}_{2}
		= \begin{cases}
			\sqrt{\sum_{j=1}^{n} \left|x_{j}'\right|^{2} } \dnb  	& 	\mathbf{x} = \mathbf{x}' \dnb , \mathbf{x}' \in \qal^{n} ,\\
			\sqrt{\sum_{j=1}^{n} \left| x_{j} \right|^{2}}	& 	\mbox{otherwise}.
		\end{cases}
	\end{equation}
	For simplicity, we refer to $\Ltwonorm{\mathbf{x}}_{2}$ as a norm, even though it is not a norm in the usual linear-algebraic sense.
	Several addition definitions are required. 
	Following \cite{Cui2023}, we write $c_{k} = \tilde{O}_{D} \left(s^{k}\right) $  for some $s < 1$ if $c_{k} = a_{k} + b_{k} \dnb \in \dnal$ and for some polynomial $h(k)$ we have $a_{k} = O \left( s^{k} h(k) \right)$ and $b_{k} = O \left(s^{k} h(k)\right)$.
	\textcite{Cui2023} proved the following result:
	\begin{theorem}[Theorem 4.1 of \cite{Cui2023}]
		\label{thm:PowerIterationConvergence}
		Let $\mathbf{u}_{1}, \dots, \mathbf{u}_{n}$ be the eigenvectors of $\mathbf{A}$ ordered so that their corresponding eigenvalues $\sigma_{j} = \alpha_{j} + \beta_{j} \dnb \in \dnal$ satisfy $\left|\alpha_{1}\right| > \left|\alpha_{2}\right| \ge \dots \ge \left|\alpha_{n} \right|$. 
		If $\mathbf{v}^{(0)} = \sum_{j=1}^{n} \mathbf{u}_{j} \gamma_{j} $ with $\gamma_{j} \in \dqal$
		such that $\gamma_{1} $ has a non-zero real coordinate, then the sequence defined by \eqref{eq:PowerIteration} satisfies
		\begin{equation*}
			\mathbf{v}^{(k)}
			= \mathbf{u}_{1} c \left( 1 + \tilde{O}_{D} \left( \left|\frac{\alpha_{2}}{\alpha_{1}}\right|^{k} \right) \right), \quad
			\lambda^{(k)}
			= \sigma_{1} \left( 1 + \tilde{O}_{D} \left( \left|\frac{\alpha_{2}}{\alpha_{1}}\right|^{2 k} \right)\right),
		\end{equation*}
		where $c = \sgn (\alpha_{1}) \frac{\gamma_{1}}{\left|\gamma_{1}\right|}$.
	\end{theorem}
	
	This theorem essentially shows that the power iteration \eqref{eq:PowerIteration} converges to $\mathbf{u}_{1}$ and its corresponding eigenvalue. At every iteration, $\mathbf{v}^{(k)}$ and $\lambda^{(k)}$ are essentially the desired values multiplied by some dual number, which approaches $1$ at an exponential rate.
	In this and in the general form of the iteration \eqref{eq:PowerIteration}, the dual quaternion power iteration is strikingly similar to the well-established power iteration for real and complex matrices \cite[Lecture 27]{Trefethen1997}.
	
	\subsection{Isometries of eigenspaces of matrices of dual quaternions}
	\label{sec:IsometriesDualQuaternion}
	
	At this point, we have established the background material and notation required to point out two important facts. 
	First, the eigenspaces of a Hermitian matrix of dual quaternion are invariant to the action of unit dual quaternions.
	Indeed, if $x \in \dqal $ and $\mathbf{x} \in \dqal^{n} $ is an eigenvector of $\mathbf{A} \in \dqmat$ corresponding to eigenvalue $\lambda \in \dnal$, then
	\begin{equation*}
		\mathbf{A} \left( \mathbf{x} x \right)
		= \mathbf{x} \lambda x
		= \left(\mathbf{x} x\right) \lambda.
	\end{equation*}
	Therefore, $\mathbf{x} x$ is also an eigenvector corresponding to eigenvalue $\lambda$.
	The last transition follows, because the $x$ and $\lambda$ commute, since the latter is in the center of $\dqal$.
	
	The second fact relates to the isometries of $\dqal^{n}$ with respect to the dual quaternion norm defined in \eqref{eq:DualQuaternionNorm}.
	Given $\mathbf{x} \in \dqal^{n}$ and an invertible $x \in \dqal$, we have
	\begin{equation*}
		\Ltwonorm{\mathbf{x} x }_{2}
		= \Ltwonorm{\mathbf{x}}_{2} \left|x\right|.
	\end{equation*}
	This can be proved by considering both cases of \eqref{eq:DualQuaternionNorm} separately and using \Cref{thm:DQAbsoluteValueProperties}.
	Therefore, $\Ltwonorm{\mathbf{x}} =\Ltwonorm{\mathbf{x} x }_{2} $ if and only $\left|x\right| = 1$, which is equivalent to $x^{*} x = 1$, i.e., $x$ is a unit dual quaternion. Therefore, the unit dual quaternions are the isometries of a free module of dual quaternions with respect to the norm structure induced by \eqref{eq:DualQuaternionNorm}.
	
	Combining these two facts together, we conclude that for a fixed unit dual number $q$, the $\dqal \to \dqal$ map $\mathbf{x} \mapsto \mathbf{x} q$ has two important properties. First, it preserves the norm \eqref{eq:DualQuaternionNorm}. Second, it maps eigenvectors of Hermitian matrices of dual quaternions to eigenvectors corresponding to the same eigenvalue.
	These two properties play a key role in the spectral approach to synchronization we develop in the next section.

	\section{The Spectral Approach to Synchronization Problems via Dual Quaternions}
	\label{sec:SpectralApproachToSynchronization}
	
	In this section, we apply the dual quaternion representation of $\SE(3)$ to the problem of $\SE(3)$ synchronization.
	In \Cref{sec:SynchronizationProblem}, we describe group synchronization problems and noise types which were considered in the literature.
	In \Cref{sec:GroupSynchronizationSpectralMethod}, we describe the spectral method, a prevalent solution to group synchronization problems.
	As we describe in \Cref{sec:DualQuaternionSynchronizationSpectralMethod}, this approach can be generalized to the dual quaternions algebra and therefore can be used to address $\SE(3)$ synchronization problems.
	We dedicate \Cref{sec:NumericalExperiments} to numerical experiments demonstrating that our proposed method indeed works on simulated data.

	\subsection{Group Synchronization: Problem Statement}
	\label{sec:SynchronizationProblem}
	
	Let $G$ be a group.
	Let $g_{1}, \dots, g_{n} $ be elements of $G$ and let $g_{i j} = g_{i} g_{j}^{-1}$ be their ratios, $ 1 \le i , j \le n$.
	A synchronization problem is an estimation problem in which one attempts to estimate $\left\{ g_{i} \setsep 1 \le i \le n \right\}$, the absolute group elements, from their ratios $\left\{ g_{i j} \setsep 1 \le i < j \le n \right\}$.
	Often, the clean ratios are perturbed by some sort of noise.
	Because $g_{1} g, \dots, g_{n} g$ generate the same set of clean measurements for every arbitrary $g \in G$, we can only hope to estimate $g_{1}, \dots, g_{n}$ up to a global right-multiplication by an arbitrary element of $G$.
	
	In order to describe the typical noise models, it is convenient to represent the group faithfully as some subgroup of the general linear group, the group of invertible matrices. 
	Thus, we assume in this section that one has an injective homomorphism $\rho : G \to \mathrm{GL}_{k} (\F[F])$, with $\F[F] = \F$ or $\F[F]  = \F[C]$.
	The absolute group elements and their ratios are mapped onto $\mathrm{GL}_{k} (\F[F])$ via $\rho$ in a consistent manner. Namely, $\rho (g_{i j}) = \rho (g_{i}) \rho (g_{j})^{-1}$.
	The elements of $G$ are represented by elements of $\F[F]^{k \times k}$, the $k \times k$ matrices with entries in $\F[F]$.
	We refer to $\F[F]^{k \times k}$ as the representation space and to the image of $\rho$ as the representing subspace.
	Throughout this subsection, we suppress the homomorphism $\rho$ and assume the absolute group elements and their ratios are already faithfully represented in some matrix algebra.
	
	Given a group represented in this manner, three kinds of noise sources are usually considered in the literature.
	First, \textbf{pertrubative noise} can be applied to each clean ratio.
	This has two forms. 
	\textbf{Multiplicative noise} perturbs the clean ratio, while still remaining in the group, so that we measure $\widehat{g}_{i j} = g_{i j} \varepsilon_{i, j}$, with $\varepsilon_{i, j} \in G$.
	\textbf{Additive noise} perturbs the clean ratio within the representation space, so that we measure $\widehat{g}_{i j} = g_{i j} + \varepsilon_{i, j}$, with $\varepsilon_{i, j} \in \F[F]^{k \times k} $.
	We here assume that pertrubative noise is either multiplicative for all ratios, or additive for all ratios.

	The last two kinds of noise are best viewed as applied to a set of measurements. This set of measurements may be assumed to be either a complete set of clean measurements $\left\{ g_{i j} \setsep 1 \le i < j \le n \right\}$, a measurement set afflicted with pertrubative noise, $\left\{\widetilde{g}_{i j} \setsep 1 \le i < j \le n \right\}$, or a subset of one of these.
	Every set of measurements can be envisaged as an undirected graph of order $n$, in which each vertex is labeled with an absolute group element and the edges correspond to the measurements in the set. See \Cref{fig:MeasurementGraph} for an example.
	
	The last two kinds of noise are best viewed as acting on the measurement graph.
	The first, which we term \textbf{selection noise}, merely removes a subset of the edges of the measurement graph. 
	The second, which is often called \textbf{corruptive} or \textbf{adversarial noise}, replaces the labels of a subset of edges with a randomly generated labels which are independent from the clean group ratios.
	In both kinds of noise, the affected subset of edges may be chosen randomly or deterministically.
	The measurement graph has a subgraph formed by edges which were not corrupted.
	We assume that this subgraph is connected, regardless of the method used to choose the subset of edges affected by selection or corruptive noise.
	\autoref{fig:MeasurementGraph} shows an example of a measurement graph with corruptive noise.
	
	\begin{figure}
		\begin{center}
			\begin{tikzpicture}[ampersand replacement=\&]
				\matrix[column sep = 1.4cm, row sep = 1.4cm, nodes={rounded corners, rectangle, draw, blue, inner sep = 1pt}] {
					\& \node(g0) {$g_{0}$}; \& \& \node(g1) {$g_{1}$}; \&\& \node(g2) {$g_{2}$}; \\
					\node(g3) {$g_{3}$};  \& \& \node(g4) {$g_{4}$}; \& \& \node(g5) {$g_{5}$}; \& \& \node(g6) {$g_{6}$}; \& \\
					\& \node(g7) {$g_{7}$}; \& \& \node(g8) {$g_{8}$}; \&\& \node(g9) {$g_{9}$}; \\
				};
				\begin{scope}[nodes = {inner sep = 1pt, font=\scriptsize}]
					\draw[darkgreen] (g0) -- node[above] {$g_{0 1}$} (g1);
					\draw[darkgreen] (g0) -- node[above left] {$g_{0 3}$} (g3);
					\draw[darkgreen] (g0) -- node[below left] {$g_{0 4}$} (g4);
					\draw (g0) -- node[near start, above] {$g_{0 5}$} (g5);
					
					\draw[darkgreen] (g1) -- node[near start, below right] {$g_{1 4}$} (g4);
					
					\draw[darkgreen] (g2) -- node[above left] {$g_{2 5}$} (g5);
					\draw[darkgreen] (g2) -- node[above right] {$g_{2 6}$} (g6);
					\draw (g2) -- node[right] {$g_{2 9}$} (g9);
					
					\draw[darkgreen] (g3) -- node[below] {$g_{3 4}$} (g4);
					\draw[darkgreen] (g3) -- node[below left] {$g_{3 7}$} (g7);
					
					\draw[darkgreen] (g4) -- node[above] {$g_{4 5}$} (g5);
					\draw[darkgreen] (g4) -- node[near start, below left] {$g_{4 8}$} (g8);
					\draw (g4) -- node[near end, below left] {$g_{4 9}$} (g9);
					
					\draw[darkgreen] (g5) -- node[near end, below right] {$g_{5 7}$} (g7);
					\draw[darkgreen] (g5) -- node[near start, below right] {$g_{5 8}$} (g8);
					\draw[darkgreen] (g5) -- node[above right] {$g_{5 9}$} (g9);
					
					\draw[darkgreen] (g6) -- node[near start, below right] {$g_{6 9}$} (g9);
				\end{scope}
			\end{tikzpicture}
			
			\caption{An example synchronization measurement graph with $10$ absolute group elements, labeled from $0$ to $9$. Vertices appear in blue circles. Edges are labeled with ratio measurements. Edges of the subgraph of non-corrupted measurements are dyed green. Edge labels of this subgraph can be thought of as either clean or afflicted with perturbative noise. In a synchronization problem, the goal is to estimate the absolute group elements in a manner that is as consistent as possible with the measurements of their ratios.}
			
			\label{fig:MeasurementGraph}
		\end{center}
	\end{figure}
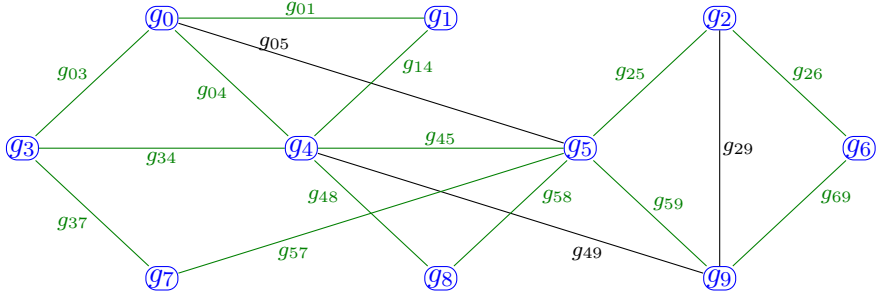
	
	All combinations of the three kinds of noise were considered in the literature, both in numerical investigations of estimation methods and in their theoretical analysis \cite{Arrigoni2016,Boumal2016,Boumal2014,Janco2022,Lerman2021,Romanov2020,Singer2011,Wang2013}. The literature cited provides a mere illustrative selection.

	\subsection{Group Synchronization via the Spectral Method}
	\label{sec:GroupSynchronizationSpectralMethod}
	
	Let $g_{1}, \dots, g_{n} $ be elements of a group $G$.
	We emphasize at the outset of this subsection that all elements of $G$ are faithfully represented within some matrix algebra $\F[F]^{k \times k }$. 
	Whenever perturbative noise is considered, it is always assumed to be multiplicative. 
	
	The spectral method relates synchronization problems to eigenproblems on a matrix, which is essentially a generalized adjacency matrix of the synchronization graph. 
	Let $\mathbf{Y} \in \F[F]^{n k \times n k}$ be a block matrix such that the $(i, j)$th block $\mathbf{Y}_{i j}$, $i < j$, is the measurement of the ratio $g_{i} g_{j}^{-1}$. 
	If the measurement is clean, perturbed (multiplicatively) or corrupted, we assign $\mathbf{Y}_{j i} = \mathbf{Y}_{i j}^{-1}$.
	If the measurement is missing, we set $\mathbf{Y}_{i j} = \mathbf{Y}_{j i} = \mathbf{0}_{ k \times k}$.
	The diagonal blocks are $\mathbf{Y}_{i i} = \mathbf{I}_{k}$ for $i=1, \dots, n$.
	Essentially, the block on the $i$th row of blocks and $j$th column of blocks is a measurement of the ratio of $g_{i} g_{j}^{-1}$.
	Since these measurements label the edges of the synchronization graph, this is a generalization of the notion of the adjacency matrix of a graph.
	
	In the absence of noise, it is evident that $\mathbf{Y} = \mathbf{g} \mathbf{g}^{-1}$.
	Here, $\mathbf{g} \in \F[F]^{n k \times k}$ is the column block matrix with $g_{i}$ as its $i$th block and $\mathbf{g}^{-1} \in \F[F]^{k \times n k}$ is the row block matrix with $g_{i}^{-1}$ as its $i$th block.
	It is immediate that $\mathbf{Y} \mathbf{g} = n \mathbf{g}$ and that $n$ is the only non-zero eigenvalue of $\mathbf{Y}$.
	Thus, the columns of $\mathbf{g}$ form a basis of the eigenspace of $\mathbf{Y}$ associated with $n$.
	Furthermore, the kernel of $\mathbf{Y}$ is the orthogonal complement of the space spanned by the columns of $\left(\mathbf{g}^{-1}\right)^{\top} \in \F[F]^{n k \times k}$. 
	Thus, $\F[F]^{n}$ is the direct sum of the kernel of $\mathbf{Y} $ and the eigenspace of $\mathbf{Y}$ associated with $n$.
	Overall, in the absence of noise, the eigenspaces structure of the measurement matrix $\mathbf{Y}$ contains information on $\mathbf{g}$, and therefore contains information on the absolute group elements.
	
	In the presence of noise, the measurement matrix $\mathbf{Y}$ can be viewed as a disrupted form of an hypothetical, clean measurement matrix of the form we described in the previous paragraph. 
	The eigenspaces of the measurement matrix can also be viewed as a disruptions of the eigenspaces of this hypothetical matrix.
	If the disruption is mild enough, the eigenspaces of the top $k$ eigenvalues of $\mathbf{Y}$ is expected to preserve enough of the information of the eigenspace associated with the eigenvalue $n$ of this hypothetical matrix.
	By projecting these top eigenvectors of the measurement matrix onto the representation of $G$, it is possible to obtain an estimate of the absolute group elements. This step may also somewhat alleviate the disruption of the eigenspace caused by the noise.
	The resulting estimation method can be formulated as a two step procedure:
	\begin{enumerate}
		\item Find a set of unit eigenvectors of $\mathbf{Y}$ associated with its largest $k$ eigenvectors. These can be placed in a matrix $\mathbf{U} \in \F^{n k \times k}$.
		
		\item Project these eigenvectors onto $G$.
		
	\end{enumerate}
	
	Implementing this two step method is fairly straightforward using existing numerical algebra libraries.
	The first step is can be easily implemented using all modern numerical eigenproblem solvers. 
	The second step, often referred to as ``\textbf{rounding}'' in the synchronization literature \cite{Romanov2020,Singer2011}, can be more involved, depending largely on $G$ and its chosen representation.
	A few examples will serve to illustrate the range in the complexity of this task.
	
	\begin{example}
		\label{exm:SO2Synchronization}
		If $G = \SO(2)$ is represented as unit complex numbers, the result of the first step is a unit vector of complex numbers.
		Simply dividing each of its entries by its complex norm yields a vector of unit complex numbers.
		This is possible, provided the disruption caused by the noise is mild enough so as to ensure all entries of the eigenvector are non-zero. 
		This rounding procedure was used to great effect in $\SO(2)$ synchronization, e.g. \cite{Perry2018,Singer2011}.
	\end{example}
	
	\begin{example}
		\label{exm:SOSync}
		If $G = \SO(k) $ is represented as real orthogonal $k \times k$ matrices, the result of the first step is a $n k \times k$ matrix which can be treated as a column block matrix with $k \times k$ blocks.
		The projection is carried out block-wise.
		The SVD of each block is calculated, yielding a decomposition $\mathbf{B} = \mathbf{U} \Sigma \mathbf{V}^{\top}$, where $\mathbf{U}$ and $\mathbf{V}$ are orthogonal matrices and $\Sigma$ is a diagonal matrix.
		The rounding step is completed by replacing every block with the image of the map $\mathbf{B} \mapsto \mathbf{U} \mathrm{diag} \left(1, \dots, 1, \det \left(\mathbf{U} \mathbf{V}^{\top} \right)\right) \mathbf{V}^{\top}$.
	\end{example}
	
	\begin{example}
		\label{exm:SE3Synchronization}
		In \cite{Arrigoni2016}, $G = \SE(3)$ was represented on the real $4 \times 4$ matrices of the form 
		\begin{equation}
			\label{eq:SE3MatrixRepresentation}
			\left[\begin{matrix}
				\mathbf{R} 	& 	\mathbf{t} \\
				\mathbf{0}^{\top} 	& 1
			\end{matrix}\right].
		\end{equation}
		This matrix represents $\left(R, \mathbf{t}\right) \in \SE(3)$, where $R$ is a rotation represented as $\mathbf{R}$, a real $3 \times 3$ orthogonal matrix with determinant of  $1$.
		Consider the matrix $\mathbf{V} \in \F^{4 n \times 4}$, the columns of which are the top eigenvectors of the measurement matrix which were calculated in the first step.
		The rounding procedure considered by \cite{Arrigoni2016} begins by extracting every fourth row of $\mathbf{V}$. Denoting the resulting $\F^{n \times 4}$ matrix by $\mathbf{V}_{4}$, all solutions of $\mathbf{V}_{4} \mathbf{u} = \mathbf{0}$ are found by finding an orthonormal basis $\mathbf{u}_{1}, \mathbf{u}_{2}, \mathbf{u}_{3}$ for the kernel of $\mathbf{V}_{4}$. A unit vector solution for $\mathbf{V}_{4} \mathbf{u}_{4} = \left(1, \dots, 1 \right)^{\top}$ is also found.
		Since nothing guarantees that solutions to these equations even exist, they are approximated in the least squares sense.
		Let $\mathbf{U} \in \F^{4 \times 4}$ be the matrix with columns $\mathbf{u}_{1}, \mathbf{u}_{2}, \mathbf{u}_{3}, \mathbf{u}_{4}$.
		Let $\mathbf{V}' = \mathbf{V} \mathbf{U}$ and substitute every forth row of the resulting matrix with $\left(0, 0, 0, 1\right)$.
		Now, each of the $4 \times 4$ blocks of $\mathbf{V}'$ has almost the form \eqref{eq:SE3MatrixRepresentation}.
		If this matrix is complex, zero out the imaginary part of every entry. 
		Finally, the top $3 \times 3$ submatrix of every block is projected onto $\SO(3)$ by utilizing its SVD in the manner described in \Cref{exm:SOSync}.
	\end{example}
	
	Non-linear transformations are employed in all three examples.
	However, the simplicity of the rounding step in \Cref{exm:SO2Synchronization} for $\SO(2)$ is contrasted by the multi-step, complicated procedure shown in the \Cref{exm:SE3Synchronization} for $\SE(3)$, the focus of this paper.
	This difference stems from inherent features of the representation of the group.
	The group $\SO(2)$ is embedded within the complex plane in a way which respects the algebraic properties of the representation space. 
	First, the norm of a complex number is directly related to the conjugation operation, i.e., $\left|z\right| = \sqrt{z z^{*}}$ for any $z \in \F[C]$. The submanifold of the complex plane into which $\SO(2)$ is embedded, the unit circle, is entirely defined in terms of this conjugation operation, i.e., $\F[T] = \left\{ z \in \F[C] \setsep z z^{*} = 1 \right\} $.
	Second, the embedding takes into consideration the $*$-algebra structure of the complex plane. The group inverse is the conjugation operation.
	
	These features do not hold for the matrix representation \eqref{eq:SE3MatrixRepresentation} of $\SE(3)$.
	As we discuss in greater detail in \Cref{sec:RepresentingSE3Comparison}, this representation merely embeds $\SE(3)$ within the multiplicative group of the matrix algebra of real $4 \times 4$ matrices.
	It does not rely or respect the other algebraic structures of the representation space. Merely transposing \eqref{eq:SE3MatrixRepresentation} almost always yields a matrix which does not represent an element of $\SE(3)$.
	
	We now discuss the nature of the rounding procedure.
	On the surface, all it does is ensure that the output of the spectral method are elements of the group, but a closer look reveals a more complicated situation.
	In the noiseless case, which is simplest, we saw that the measurement matrix has a single non-zero eigenvalue. 
	We consider now the linear isometries of its eigenspace with respect to the standard Euclidean norm.
	These isometries form a compact group.
	Numerical linear algebra tools ensure we shall obtain unit eigenvectors of the measurement matrix. These may be $\frac{1}{n} \mathbf{g}$, where the $\frac{1}{n}$ term ensures these columns are unit eigenvectors. However, more often they are expected to be $\frac{1}{n} \mathbf{g} \mathbf{S}$, where $\mathbf{S}$ is an isometry of the eigenspace of the non-zero eigenvalue of the measurement matrix.
	
	Let us consider the three examples above. In $\SO(2)$ synchronization (\Cref{exm:SO2Synchronization}), the group of isometries is $\SO(2)$ itself, represented by unit complex numbers.
	Indeed, since the eigenspace is one-dimensional, every isometry has the form $\mathbf{x} \mapsto \mathbf{x} x$ for $x \in \F[C]$. By substituting $\mathbf{x} x$ into the standard Euclidean norm, it is easy to show that it is an isometry if and only if $\Ltwonorm{\mathbf{x} x} = \Ltwonorm{x} \left|x\right|$, which holds if and only if $\left|x\right| = 1$. Therefore, $x \in \SO(2)$.
	It follows that in the noiseless case, the spectral method recovers the absolute group elements up to a right-multiplication by an element of $\SO(2)$. And so, in the presence of noise, the rounding step alleviates a disruption in the eigenspace structure and nothing more.
	
	This is not the case in the other two examples. There, the rounding has to perform other tasks, which are unrelated to the disruption  of the eigenspaces caused by the noise and will be required even in its absence.
	Again, it is instructive to consider the noiseless case first.
	In $\SO(k)$ synchronization (\Cref{exm:SOSync}), the non-zero eigenspace is a $k$-dimensional subspace of $\F^{n}$ and its isometry group with respect to the standard Euclidean norm on $\F^{n}$ is $\OG(k)$, the group of orthogonal transformations of $\F^{k}$. The group $\SO(k)$ is a proper subgroup of $\OG(k)$ and so the unit eigenvectors in the eigenspace can be separated into two disjoint populations. The first population is obtained by applying $\SO(k)$ to the columns of $\frac{1}{n} \mathbf{g}$ from the right. The second can be obtained by applying elements of $\OG(k) \setminus \SO(k)$.
	This means that being a solution to the synchronization problem implies being a basis of an eigenspace of the non-zero eigenvalue of the measurement matrix, but the converse does not hold.
	We refer to this phenomena as the \textbf{eigenspace-synchronization gap}.
	Therefore, in the presence of noise, the rounding procedure performs at least the following two tasks. It alleviates the disruption to the eigenspace structure \textit{and} ensures that the resulting estimate of the absolute group elements is indeed the absolute group elements up to a multiplication by an arbitrary element of $G$ from the right. The latter can no longer be taken for granted.
	
	These issues are exacerbated considerably in $\SE(3)$ synchronization (\Cref{exm:SE3Synchronization}), where the group is non-compact. There, even in the absence of noise, the group $\SE(3)$ is not even a subgroup of the isometries of the eigenspace. 
	This is the case because the group of isometries is compact, while $\SE(3)$ is not.
	In this case, the eigenspace-synchronization gap is not so much a gap as an abyss.
	The domain of the rounding procedure is a compact manifold, matrices formed with four unit eigenvectors as columns. Its codomain is a non-compact manifold, block matrices with blocks of the form \eqref{eq:SE3MatrixRepresentation}.
	The rounding procedure maps this compact domain to this non-compact codomain.
	Thus, even in the absence of noise, for $\SE(3)$ synchronization it is another large step away from the neat correspondence exhibited in  $\SO(2)$ synchronization (\Cref{exm:SO2Synchronization}) between eigenvectors of the measurement matrix and estimates of the absolute group elements. In the presence of noise, it does a great deal more than handling the eigenspace disruption. 
	
	In closing this subsection, we digress to note an important distinction between the way we introduced the spectral method and the way it is often discussed in the literature. In the literature on synchronization, the relationship between the eigenvectors of the noisy measurement matrix and an estimate of the absolute group elements is often arrived at using a optimization terminology. 
	For instance, the initial goal of both \cite{Arrigoni2016} and \cite{Singer2011} is an estimate the absolute group elements obtained by solving a certain least squares problem over $\widehat{\mathbf{g}} \in G^{n}$. Here, $G^{n}$ is the set of all column block matrices with blocks representing elements of $G$.
	Since the domain of the problem is non-convex, this is a challenging optimization problem to solve in practice. This constraint is therefore changed to all matrices satisfying $\widehat{\mathbf{g}}^{*} \widehat{\mathbf{g}} = n \mathbf{I}_{k}$.
	In the case of \cite{Singer2011}, which worked on $\SO(2)$ represented on the unit complex numbers, this change in the constraint is a relaxation of the constraint of the original least squares problem. That is, the feasible set of the original problem is replaced by a larger set.
	However, this change is not a relaxation of a constraint in \cite{Arrigoni2016}, which worked on $\SE(3)$ with the matrix representation \eqref{eq:SE3MatrixRepresentation}.
	Regardless of that, in both cases, the new least squares problem turns out to be easily solved by finding the eigenvectors of the measurement matrix.
	
	We note this here merely to point out that the optimization perspective provides a useful heuristic for deriving practically useful estimation methods  for synchronization problems. 
	However, it does not offer an immediate way to prove when and why these methods work.
	We here treat the noisy measurement matrix as a disrupted version of a clean, complete measurement matrix, and seek to approximate the eigenspaces of the latter using the eigenspaces of the former.
	In the optimization perspective, the very same eigenproblem is considered as a relaxed optimization problem, which one hopes approximates the solution of the non-convex problem.
	This is a difference without much consequence to the governing mathematics of the spectral method. Regardless of the heuristic used to derive the method, considerable effort is required to come up with and prove theoretical guarantees on the performance of the spectral method.
	Indeed, there is a rich literature of formal characterizations of when and how solutions of the eigenproblem provide good estimates of the absolute orientations, e.g., \cite{Ling2022,Ling2022a,Romanov2020,Singer2011} with many mathematical tools, including random matrix theory and perturbation theory, being brought to bear on these questions.

	\subsection{A Spectral Approach to $\SE(3)$ Synchronization via Dual Quaternion Representation}
	\label{sec:DualQuaternionSynchronizationSpectralMethod}
	
	In \Cref{sec:GroupSynchronizationSpectralMethod}, we discussed the contrast between the rounding step in $\SO(2)$ and $\SE(3)$ synchronization in \Cref{exm:SO2Synchronization} and \Cref{exm:SE3Synchronization}.
	We identified two types of differences between them. First, their respective representations have different features. For instance, the representation of $\SO(2)$ agreed with involutive algebraic structure of the space it was embedded in, whereas the representation of $\SE(3)$ did not.
	Second, in a sense, the rounding does more in the $\SE(3)$ case than in the $\SO(2)$ case.
	We now utilize the material we surveyed in \Cref{sec:DualQuaternionAlgebraAndFreeModules} to develop a spectral synchronization method for $\SE(3)$, which is more similar to the spectral method for $\SO(2)$ in both these respects.
	We note in passing that the quaternions and the dual quaternions were used in the past to address synchronization problems \cite{Govindu2001,Torsello2011}, but this is the first time an analog of the spectral method for synchronization was developed using dual quaternions.
	
	All the necessary components of the spectral method are present over the dual quaternions. 
	Elements of $\SE(3)$ are represented as unit dual quaternions, as we surveyed in \Cref{sec:SE3RepresentationOnDualQuaternions}.
	This representation agrees with both the multiplication operation of the algebra and its involution, since the dual quaternion conjugate of unit dual quaternion is its inverse.
	The generalized adjacency matrix of the synchronization graph is a matrix of dual quaternions, rather than a block matrix with real or complex entries. 
	The resulting matrix is Hermitian, in the dual quaternion sense.
	By the spectral theorem for dual quaternion matrices, it has a eigenvector decomposition, as described in \Cref{sec:DQSpectralTheorem}.
	
	Now, given a clean, complete set of measurements, the measurement matrix has the form $\mathbf{Y} = \mathbf{g} \mathbf{g}^{*}$, where $\mathbf{g} = \left(g_{1}, \dots, g_{n}\right)^{\top}$ is a column vector of unit dual quaternions and $\mathbf{g}^{*}$ is its conjugate transpose.
	In the presence of noise, we consider the dual quaternion eigenvector of the measurement matrix corresponding to its top eigenvalue as an approximation of $\mathbf{g}$, the eigenvector of the clean measurement matrix.
	Here, ``top'' is with respect to the order of eigenvalues established in \Cref{thm:SpectralTheorem}.
	This eigenvector can be approximated using the power method we surveyed in \Cref{sec:PowerIteration}.
	The rounding is carried out by applying the map defined in \Cref{thm:ProjectionSE3} to each entry this eigenvector. This rounding procedure stems directly from the close relationship between the algebraic structure and geometry of the dual quaternion algebra and the unit dual quaternions, as we discussed at the end of \Cref{sec:SE3RepresentationOnDualQuaternions}.
	If the noise is mild enough, all entries of the top eigenvector are expected to be invertible, in which case the projection amounts to applying the map $q \mapsto \frac{q}{\left|q\right|}$ entry-wise to the eigenvector.
	
	Overall, we obtain the following two-step algorithm:
	\begin{enumerate}
		\item Find the eigenvector of $\mathbf{Y}$ corresponding to its largest eigenvalue, in the order established in \Cref{thm:PowerIterationConvergence}.
		Use the power iteration we surveyed in \Cref{sec:PowerIteration}.
		Denote this eigenvector by $\widehat{g} = \left(\widehat{g}_{1}, \dots, \widehat{g}_{n}\right)^{\top}$.
		
		\item Carry out the substitution $\widehat{g}_{j} \to \frac{\widehat{g}_{j}}{\left|g_{j} \right|} $ on each of its entries ($j=1,\dots, n$).
		
	\end{enumerate}
	As we demonstrate in \Cref{sec:NumericalExperiments}, this algorithm yields comparable performance to the state-of-the-art spectral approach to $\SE(3)$ synchronization described in \cite{Arrigoni2016}.
	We emphasize the simplicity of the rounding step compared to the one employed by \cite{Arrigoni2016}, which we described in \Cref{exm:SE3Synchronization}. Its simplicity is owed entirely to the close relationship between $\SE(3)$ and the algebra of dual quaternions.
	We also emphasize that here the rounding procedure does nothing more than handle the disruption the noise causes to the eigenspaces of the clean measurement matrix. Indeed, as we discussed in \Cref{sec:IsometriesDualQuaternion}, the isometry group of the eigenspaces of $\mathbf{Y}$ is actually the unit dual quaternions and as we discussed in \Cref{sec:SE3RepresentationOnDualQuaternions}, each unit dual quaternion represents an element of $\SE(3)$. Thus, there is no eigenspace-synchronization gap, despite the non-compact setting. Our method enjoys similar performance to the current state-of-the-art in spectral synchronization for $\SE(3)$, yet reaps the benefits of embedding $\SE(3)$ in an algebra that is much closer to it in structure.
	
	Taken together, the resulting spectral method is remarkably similar to the spectral method of $\SO(2)$ we showcased in \Cref{exm:SO2Synchronization}. 
	This is evident in two places. First, in the properties of the representing space and its relationship to the subset which represents the group. Second, in the underlying tasks performed by the rounding step.
	This similarity is especially remarkable when we consider the fact that $\SO(2)$ is a compact, commutative group, while $\SE(3)$ is a non-compact and non-commutative group.
	
	Finally, we note that \textcite{Cui2023} addressed a synchronization problem with selection noise, though they do not use the word synchronization at all. 
	They proposed an iterative method, which is distinct from the spectral method we constructed here, both in its form and its theoretical underpinnings. 
	First, it is not the spectral method as we conceived it here. At every iteration, the iteration \eqref{eq:PowerIteration} is used to approximate the eigenvector of some matrix of dual quaternions.
	Second, their iterative method was motivated by the need to solve an explicit optimization problem over the dual quaternions, whereas we relied on the heuristic picture of the eigenspaces of a matrix being disrupted by noise.

	\section{Numerical Experiments}
	\label{sec:NumericalExperiments}
	
	We run several numerical experiments demonstrating the utility of the dual quaternions in synchronization problems.
	We begin in \Cref{sec:SyntheticEXperimentsSetup} by describing how the synthetic synchronization measurement matrices were simulated in our experiments and describe our experimental procedure.
	In \Cref{sec:ResponseToMissingEntries} and \Cref{sec:ResponseToCorruptive} we show and discuss the results of the synthetic data experiments. They demonstrate that using the dual quaternions to represent $\SE(3)$ yields comparable estimates to the current state-of-the-art spectral synchronization method \cite{Arrigoni2016}, while enjoying the simpler rounding procedure obtained by embedding $\SE(3)$ in the algebra of dual quaternions.

	\subsection{Synthetic Experiments Setup}
	\label{sec:SyntheticEXperimentsSetup}
	
	We conducted several experiments on synthetic measurements matrices, modeled after the synthetic data experiments carried out in \cite{Arrigoni2016}, the current state-of-the-art in $\SE(3)$ synchronization.
	Probability functions on $\SE(3)$ were defined using the angle-axis-translation representation discussed in \Cref{sec:RepresentingSE3Comparison}.
	Ground truth elements $g_{1}, \dots, g_{n} \in \SE(3)$ were drawn from i.i.d.\ random variables. 
	The angles were independently and uniformly distributed over $[0, 2 \pi )$.
	The axis was independently and uniformly distributed over $S^{2}$.
	Translations of the form $\mathbf{t} = \left(t_{1}, t_{2}, t_{3} \right)^{\top} \in \F^{3}$ were independently distributed with $t_{1}$, $t_{2}$ and $t_{3}$ being i.i.d. standard Gaussian random variables.
	The measurement matrix $\mathbf{Y} \in \dqmat $ follows the following model:
	\begin{equation}
		\label{eq:MeasurementModel}
		\mathbf{Y}_{i j} 
		= \begin{cases}
			e_{i j} \left( s_{i j} g_{i} g_{j}^{-1}  \varepsilon_{i, j}  + \left(1 - s_{i j}\right) c_{i j}\right)  & 1 \le i < j \le n, \\
			1 	& i = j ,\\
			\mathbf{Y}_{j i}^{*} 	& 1 \le j < i \le n.
		\end{cases}
	\end{equation}
	Here, $\left(e_{i j} \setsep 1 \le i < j \le n \right)$ are i.i.d Bernoulli random variables with parameter $p$ and 
	$\left(s_{i j} \setsep 1 \le i < j \le n \right)$ are i.i.d. are the same with parameter $q$.
	The former are indicators of existing measurements and the latter are indicators of entries which were not corrupted.
	Denoting by $\circ$ the entry-wise product of two matrices of elements of $\SE(3)$, the matrices $\mathbf{E} = \left(e_{i j}\right)$, $\mathbf{S} = \left(s_{ i j}\right)$ and $\mathbf{E} \circ \mathbf{S} = \left(e_{i j}s_{i j}\right)$ are all the adjacency matrices of random graphs following the Erd\H{o}s-R\'{e}nyi model with parameter $p$, $q$ and $p q$, respectively.
	The corrupted entries $\left(c_{i j} \setsep 1 \le i < j \le n\right)$ are i.i.d. with an angle in degrees normally distributed with zero mean and variance $\sigma_{r}^{2}$, axis distributed uniformly over $S^{2}$ and translation with $t_{1}, t_{2}, t_{3}$ being i.i.d. Gaussian with zero mean and variance $\sigma_{t}^{2}$.
	All random variables mentioned above were also independent of each other.
	
	Given a ground truth $\mathbf{g} \in \SE(3)^{n}$ and an estimate $\widehat{\mathbf{g}} \in \SE(3)^{n}$, we measure the quality of the estimate as follows.
	Due to the symmetry in the observations, we have to align $\mathbf{g}$ and $\widehat{\mathbf{g}}$ in the following sense. We find a $g \in \SE(3)$ such that a distance metric between $\mathbf{g}$ and $\widehat{\mathbf{g}} g$ is minimized, where $\widehat{\mathbf{g}} g$ is the result of multiplying by $g$ from the right every entry of $\widehat{\mathbf{g}}$.
	In \Cref{apx:Alignment}, we discuss the specifics, including the distance metric we used.
	Here, we describe the procedure itself.
	First, we assume that every element of $\SE(3)$ is represented by $(q, \mathbf{t})$ with $q \in \qal$ being a unit quaternion with non-negative first coordinate as discussed in \Cref{sec:SO3RepresentationOnQuaternions} and $\mathbf{t} \in \F^{3}$.
	Second, if $g = (q, \mathbf{t})$, 
	$\mathbf{g} = \left(\left(q_{1}, \mathbf{t}_{1}\right), \dots, \left(q_{n}, \mathbf{t}_{n}\right)\right)$ and 
	$\widehat{\mathbf{g}} = \left(\left(\widehat{q}_{1}, \widehat{\mathbf{t}}_{1}\right), \dots, \left(\widehat{q}_{n}, \widehat{\mathbf{t}}_{n}\right)\right)$, we have
	\begin{equation}
		\label{eq:BestAlignerRotation}
		q 
		= \frac{s}{\left|s\right|}, \qquad
		s
		= \sum_{j=1}^{n} \widehat{q}_{j}^{*} q_{j},
	\end{equation}
	and 
	\begin{equation}
		\label{eq:BestAlignerTranslation}
		\mathbf{t}
		= \frac{1}{n} \sum_{j=1}^{n} \varphi \left(\widehat{q}_{j}^{*}\right) \left( \mathbf{t}_{j} - \widehat{\mathbf{t}}_{j}\right).
	\end{equation}
	This alignment procedure was used on the estimate obtained by both methods employed in our experiments, our dual quaternion spectral synchronization and the spectral method of \cite{Arrigoni2016}, which used the matrix representation \eqref{eq:SE3MatrixRepresentation}.
	We note that our alignment method is considerably simpler to apply than the alignment method utilized by \cite{Arrigoni2016}, whose numerical experiments inspired our own.
	
	All experiments were repeated fifty times, each with freshly generated ground truth and noise realization with the specified experimental parameters.
	From each estimate we extracted the entry-wise error, after alignment, for the rotations and translations, separately.
	For the rotations, the error was calculated as
	$d\left(q_{1}, q_{2}\right)
	= 2 \arccos \left( 2 \Eucprod{q_{1}}{q_{2}}^{2} - 1\right)$, 
	where $\Eucprod{q_{1}}{q_{2}}$ is the Euclidean inner product applied to $q_{1}$ and $q_{2}$ considered as elements of $\F^{4}$ with the obvious embedding.
	See \Cref{rem:RotationDistanceMetric} at the end of \Cref{apx:Alignment} for an explanation of this metric.
	For the translations, we used the Euclidean distance, 
	$d \left(\mathbf{t}_{1}, \mathbf{t}_{2} \right)
	= \Ltwonorm{\mathbf{t}_{1} - \mathbf{t}_{2}}$.
	For each, the mean, minimum and maximum error were calculated and each of these measures was averaged over the fifty repeats. The plotted data for every quantity is always its mean over the repeats.
	
	We described the approach of \cite{Arrigoni2016} in \Cref{exm:SE3Synchronization}.
	Both estimation methods were implemented in Python.
	\textcite{Arrigoni2016} originally implemented their approach in MATLAB.
	We took care to use equivalent functions to theirs in the Python ecosystem.
	Code used to simulate the figures below is available at \hyperref{https://github.com/idohadi/dqsync-python/}{}{}{https://github.com/idohadi/dqsync-python/}.

	\subsection{Response to Missing Entries and Perturbative Noise}
	\label{sec:ResponseToMissingEntries}
	
	In \Cref{fig:SelectionPerturbativeSyntheticRotationNoise}, we show the results with increasing levels of multiplicative noise on the rotations and no noise on the translations.
	In \Cref{fig:SelectionPerturbativeSyntheticTranslationNoise}, we show the results with increasing levels of multiplicative noise on the translations and no noise on the rotations.
	Consistent with the semidirect product structure of $\SE(3)$, in \Cref{fig:SelectionPerturbativeSyntheticRotationNoise} the estimate of the translational component is also affected by noise, whereas in \Cref{fig:SelectionPerturbativeSyntheticTranslationNoise} the quality of the estimate of the rotations is basically flat, unaffected by the noise.
	When we vary the level of noise on both the rotational and translational part in \Cref{fig:SelectionPerturbativeSyntheticRotationTranaslationNoise}, the error rises on both components.
	Overall, using the dual quaternion representation of $\SE(3)$ yields marginally better mean entry-wise error, especially in the estimate of the rotation and when more entries are missing ($p = 0.05$ vs. $p = 0.3$).
	It is also evident that using the dual quaternions yields somewhat stabler estimates, as indicated by the smoother dual quaternion curves evident in \Cref{fig:MissingEntriesAndPerturbativeNoise}.
	This may be a feature of our implementation, rather than the method, as these peaks are not evident in the plots shown in \cite{Arrigoni2016}.
	All trends we discussed above are also evident in minimum and maximum entry-wise errors.
	We conclude that the two methods perform comparably and that the dual quaternions can be used effectively to represent $\SE(3)$ in applied estimation problems.
	
	\begin{figure}
		\begin{subfigure}{\textwidth}
			\centering
			\includegraphics[scale=0.6]{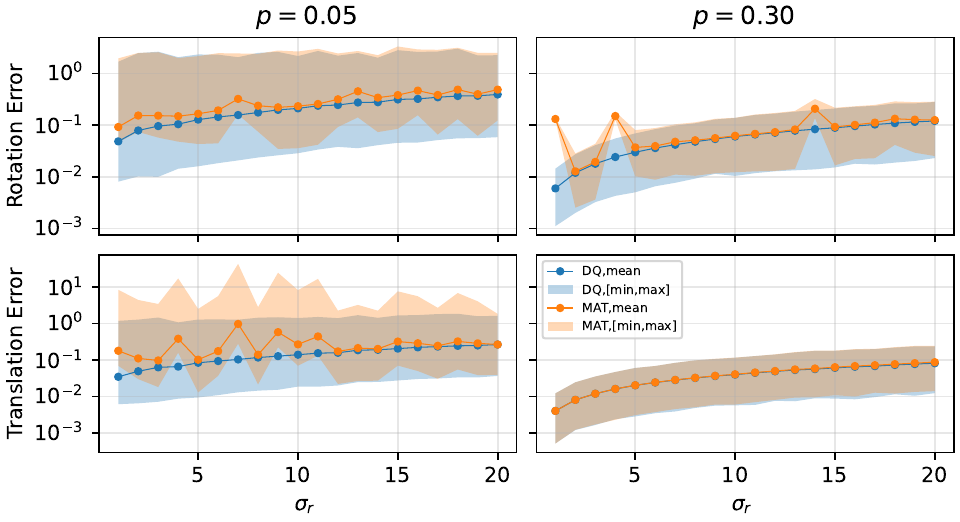}
			\caption{$\sigma_{r} = 1^{\circ}, 2^{\circ}, \dots, 20^{\circ}$ and $\sigma_{t} = 0$.}
			\label{fig:SelectionPerturbativeSyntheticRotationNoise}
		\end{subfigure}
		
		\begin{subfigure}{\textwidth}
			\centering
			\includegraphics[scale=0.6]{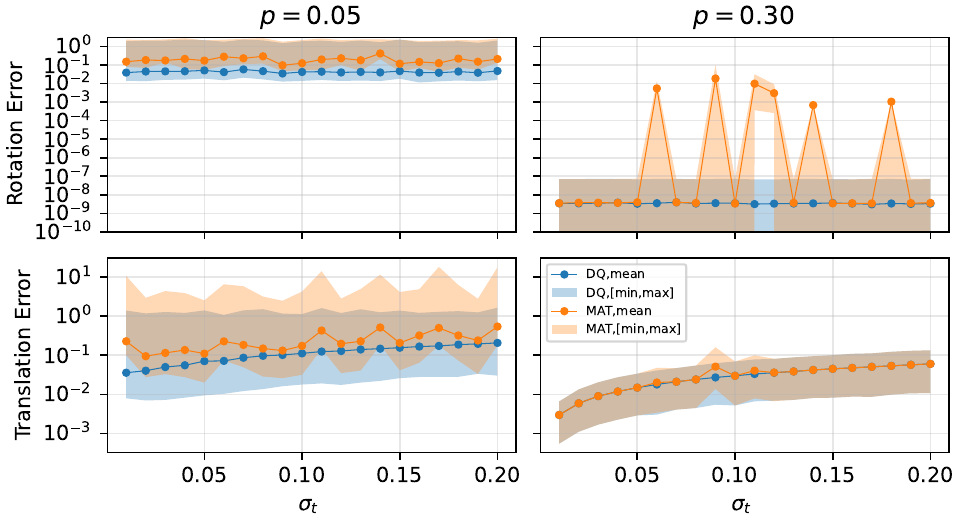}
			\caption{$\sigma_{r} = 0^{\circ}$ and $\sigma_{t} = 0.01, 0.02, \dots, 0.2$.}
			\label{fig:SelectionPerturbativeSyntheticTranslationNoise}
		\end{subfigure}
		
		\begin{subfigure}{\textwidth}
			\centering
			\includegraphics[scale=0.6]{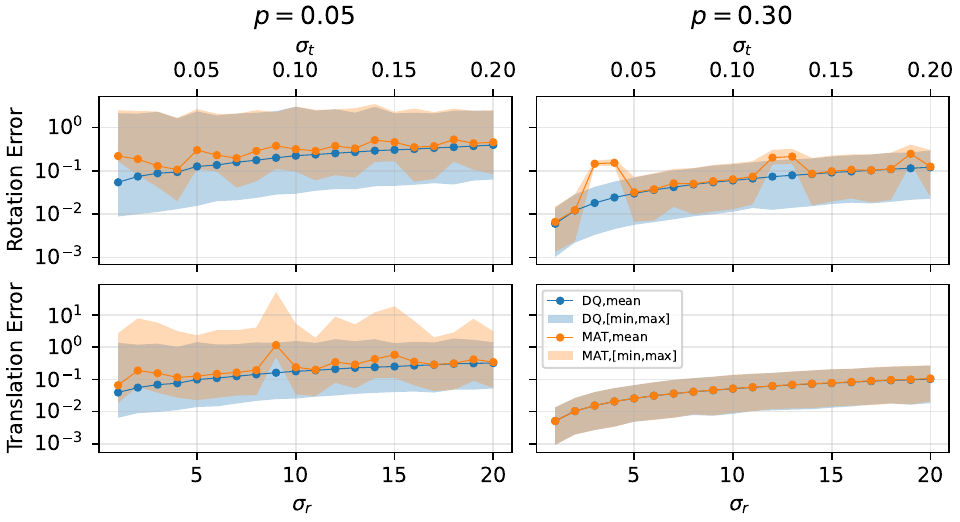}
			\caption{$\left(\sigma_{r}, \sigma_{t}\right) = \left(1^{\circ}, 0.01\right), \left(2^{\circ}, 0.02\right), \dots, \left(20^{\circ}, 0.2 \right)$.}
			\label{fig:SelectionPerturbativeSyntheticRotationTranaslationNoise}
		\end{subfigure}
		\caption{\textbf{Synthetic data experiments with missing entries and perturbative noise.} Throughout experiments, problem size is $n=100$ and there are no corrupted entries ($q = 1$). Every data point is an average over $50$ repeats. The results using the dual quaternion representation are denoted \textbf{DQ} and using the matrix representation \eqref{eq:SE3MatrixRepresentation} are denoted \textbf{MAT}.}
		\label{fig:MissingEntriesAndPerturbativeNoise}
	\end{figure}

	\subsection{Response to Corruptive Noise}
	\label{sec:ResponseToCorruptive}
	
	In \Cref{fig:CorruptivePerturbative}, we used no perturbative noise but there were missing entries.
	In \Cref{fig:CorruptivePerturbativePerturbative}, there are no missing entries but there is perturbative noise.
	The same trends are evident in both of them. 
	First, the rotation estimate is better using the matrix representation \eqref{eq:SE3MatrixRepresentation}, while the translation estimate is better using the dual quaternion representation.
	Second, the differences in estimate quality are larger in \Cref{fig:CorruptivePerturbativePerturbative} than in \Cref{fig:CorruptivePerturbative}.
	Overall, these experiments indicate no clear advantage to either method, but do underscore the fact the dual quaternions can be used effectively to represent $\SE(3)$ in applied estimation problems.

	\begin{figure}[!t]
		\begin{subfigure}{\textwidth}
			\centering
			\includegraphics[scale=0.8]{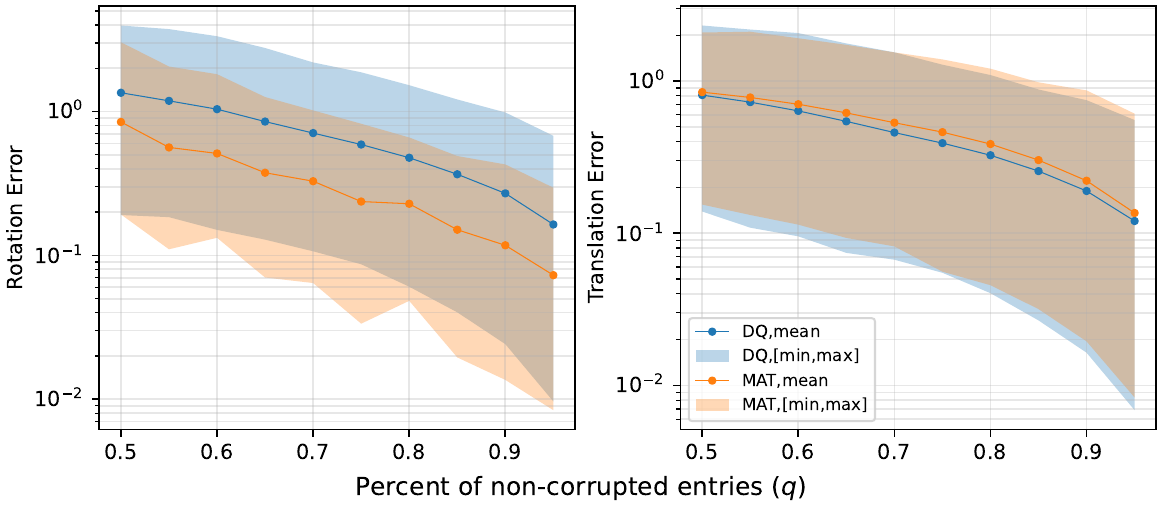}
			\caption{$q = 0.5, 0.55, \dots, 0.95$ with $\sigma_{r} = 0^{\circ}$ and $\sigma_{t} = 0$ and $p = 0.2$.}
			\label{fig:CorruptivePerturbative}
		\end{subfigure}
		
		\vspace{15pt}
		
		\begin{subfigure}{\textwidth}
			\centering
			\includegraphics[scale=0.8]{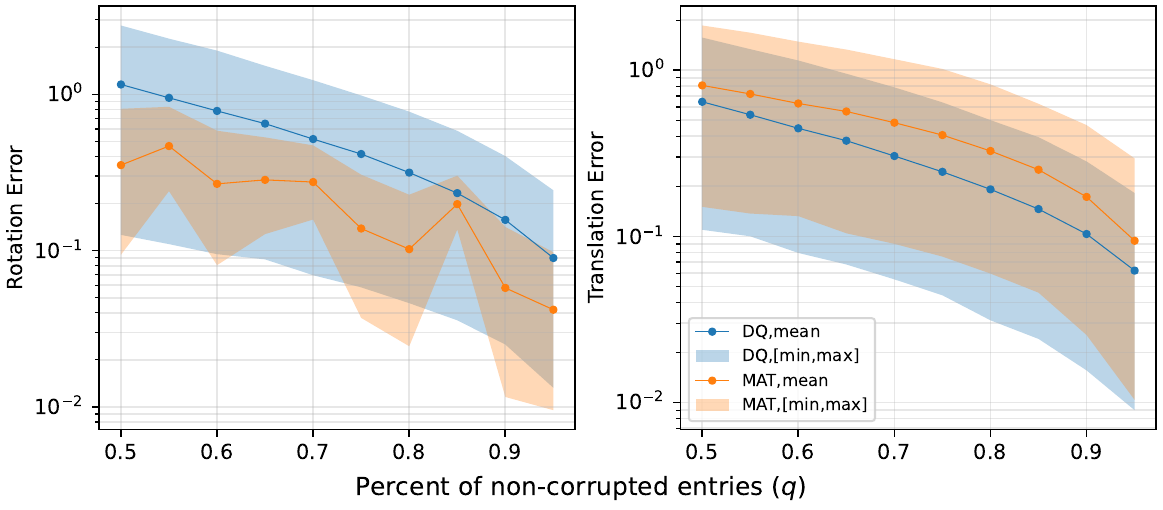}
			\caption{$q = 0.5, 0.55, \dots, 0.95$ with $\sigma_{r} = 5^{\circ}$ and $\sigma_{t} = 0.05$ and $p = 1$.}
			\label{fig:CorruptivePerturbativePerturbative}
		\end{subfigure}
		
		\caption{\textbf{Synthetic data experiments with corruptive noise.} Throughout experiments, problem size is $n=100$. Every data point is an average over $50$ repeats. The results using the dual quaternion representation are denoted \textbf{DQ} and using the matrix representation \eqref{eq:SE3MatrixRepresentation} are denoted \textbf{MAT}.}
		\label{fig:CorruptiveNoise}
	\end{figure}

	\appendix
	
	\section{Comparison of Representations of $\SE(3)$}
	\label{sec:RepresentingSE3Comparison}
	
	We compare the properties of three representations of $\SE(3)$ which were used in applied mathematics work.
	Our discussion is skewed towards properties which we think are central to how notions from linear algebra can be used to solve applied problems over $\SE(3)$.
	
	The most straightforward way to represent elements of $\SE(3)$ is as affine transformations, a pair $\left(R, \mathbf{t}\right) \in \SO(3) \ltimes \F^{3} $.
	These may come in several variants, differing in the way the rotation is represented.
	Here we survey the axis-angle and rotation matrix representations.
	
	In the axis-angle representation, a rotation $R$ is represented by the axis $\mathbf{v} \in S^{2} \subset \F^{3}$ of the rotation and the angle $\theta $ (in radians) by which vectors are rotated around $\mathbf{v}$.
	Both can be concisely encapsulated in a \textbf{rotation vector} $\mathbf{x} = \theta \mathbf{v} \in \F^{3}$, which are elements of $B_{3} \left(\mathbf{0}, 2 \pi \right)$, where $B_{3} \left(\mathbf{x}, r \right) = \left\{ \mathbf{y} \in \F^{3} \setsep \Ltwonorm{\mathbf{x} - \mathbf{y}}_{2} < r \right\}$.
	Every pair of a rotation vector and a translation can represent an element of $\SE(3)$ and one can easily extend that to pairs of the form $\F^{3} \times \F^{3}$ by taking the Euclidean norm of the rotation vector modulo $2 \pi$.
	Thus, we say that \textbf{representation space}, the space of all elements of the form of the given representation, is $\F^{3} \times \F^{3}$.
	The \textbf{representing subset}, the subset of the representation space into which $\SE(3)$ is embedded, is $B_{3} \left(\mathbf{0}, 2 \pi \right) \times \F^{3}$.
	Furthermore, there is projection from the representation space to the representation subset, mapping every element of the representation space to the representing subset.
	We refer to this representation as the \textbf{axis-angle-translation representation}.
	
	From all this it follows in particular that in this representation the dimension of the representation space is $6$, which matches the manifold dimension of $\SE(3)$, which makes it the most parsimonious representation we will survey.
	However, there is no direct way to calculate the product of two elements of $\SE(3)$ in this representation. It is necessary to convert them into one of the other representations we will survey.
	We do not know of any established way to construct something like an algebra over this representation space and consequently we know of no such way to construct matrix algebras and investigate their spectral properties.
	
	The other $\SO(3)$ representation we consider are rotation matrices, real $3 \times 3$ orthogonal matrices with a determinant of $1$.
	This representation comes with a multiplicative structure of matrix multiplication and it is embedded in the matrix algebra of all real $3 \times 3$ matrices.
	The representation space of $\SE(3)$ is then $\F^{3 \times 3} \times \F^{3} \cong \F^{12}$.
	Its multiplicative and vector space structure can be identified with the set of all affine transformations of $\F^{3}$, transformations of the form $\mathbf{x} \mapsto A \mathbf{x} + \mathbf{t}$, where $A$ is a linear transformation and $\mathbf{t} \in \F^{3}$.
	This set has both an addition of and multiplication operations, defined by adding or composing, respectively, two affine transformations.
	Despite this, it does not form a ring, because these two binary operations do not satisfy one of the distributive properties.
	Borrowing an example from \cite{Blackett1956}, suppose that $\left(A_{1}, \mathbf{t}_{1}\right)$ and $\left(\mathbf{0}, \mathbf{t}_{2}\right)$ are two affine transformations. Here, $\mathbf{0}$ is the zero linear transformation and $\mathbf{t}_{2} \ne 0$.
	We have
	\begin{equation*}
		\left(\mathbf{0}, \mathbf{t}_{2}\right) \circ \left( \left(A_{1}, \mathbf{t}_{1}\right) + \left(A_{1}, \mathbf{t}_{1}\right) \right) 
		= \left(\mathbf{0}, \mathbf{t}_{2}\right),
	\end{equation*}
	but 
	\begin{equation*}
		\left(\mathbf{0}, \mathbf{t}_{2}\right) \circ \left(A_{1}, \mathbf{t}_{1}\right) + \left(\mathbf{0}, \mathbf{t}_{2}\right) \circ \left(A_{1}, \mathbf{t}_{1}\right)
		= 2 \mathbf{t}_{2}.
	\end{equation*}
	While not being a ring, the affine transformations form a different algebraic structure, that of a near-ring.
	As far as we know, matrix spaces and spectral decompositions were not studied over near-rings in general and over the affine transformations in particular.
	Furthermore, to the best of our knowledge, a projection of affine transformations onto the representing subset has not been worked out in the literature.
	
	Finally, $\SE(3)$ has a group embedding within the matrix algebra $\F^{4 \times 4}$.
	An element $\left(R, \mathbf{t}\right) \in \SE(3)$, where $R$ is represented as $\mathbf{R}$, real $3 \times 3$ orthogonal matrix with determinant of  $1$, can be represented as a block matrix
	\begin{equation}
		\label{eq:SE3MatrixRepresentation2}
		\left[\begin{matrix}
			\mathbf{R} 	& 	\mathbf{t} \\
			\mathbf{0}^{\top} 	& 1
		\end{matrix}\right].
	\end{equation}
	Rectangular block matrix where all blocks are of this form obviously belong to the matrix algebra $\F^{4 n \times 4 n}$, which definitely has a spectral decomposition theorem and many other matrix decompositions.
	Furthermore, an invertible $\mathbf{A} \in \F^{4 \times 4}$ can also be projected onto the representing subset by a two step procedure.
	First, finding a invertible matrix $\mathbf{B} \in \F^{4 \times 4}$ such that the matrix $\mathbf{A} \mathbf{B} $ has $\left(0, 0,0, 1\right)$ in its bottom row. This amounts to changing the coordinates of the matrix $\mathbf{A}$ to one which satisfies this particular constraint on the bottom row.
	Second, denoting by $\left(\mathbf{B} \mathbf{A}\right)_{3\times 3}$ the upper left $3 \times 3$ submatrix of $\mathbf{B} \mathbf{A}$, one takes its singular value decomposition $\left(\mathbf{B} \mathbf{A}\right)_{3\times 3} = \mathbf{V} \mathbf{S} \mathbf{U}^{\top}$, where $\mathbf{S}$ is diagonal and $\mathbf{U}$ and $\mathbf{V}$ are orthogonal. 
	To obtain a matrix of the form \eqref{eq:SE3MatrixRepresentation2}, replace $\left(\mathbf{B} \mathbf{A}\right)_{3\times 3}$ with $\mathbf{V} \mathrm{diag} \left(1, 1, \det \left(\mathbf{V} \mathbf{U}\right) \right) \mathbf{V}^{\top}$.
	
	The great advantage of the matrix representation \eqref{eq:SE3MatrixRepresentation2} is the fact that its representation space is a matrix algebra.
	This allows the use of the familiar linear algebraic toolkit and in particular allows one to define the projection we described above.
	However, this is also the source of the greatest disadvantage.
	The group $\SE(3)$ is embedded in the multiplicative group of $\F^{4 \times 4}$, and no more.
	The other algebraic structures defined on $\F^{4 \times 4}$ are inconsistent with the embedding.
	Merely transposing \eqref{eq:SE3MatrixRepresentation2} would generally take one out of the representing subspace, indicating the involution of $\F^{4 \times 4}$ plays no role in representing $\SE(3)$.
	Furthermore, the relationship between the projection and group structure is hard to characterize.
	This last problem comes to the fore in the situation discussed in \Cref{sec:GroupSynchronizationSpectralMethod}, where we describe the projection onto $\SE(3)$ of an entire column block matrix with $\F^{4 \times 4}$ blocks, which is used in an actual application \cite{Arrigoni2016}.
	Perhaps underscoring the gap between $\SE(3)$ and the representation space in this case, is the fact $\SE(3)$ is a manifold of dimension $6$, while the representation space is $16$-dimensional.
	
	Consider the unit dual quaternion representation of $\SE(3)$ we surveyed in \Cref{sec:SE3RepresentationOnDualQuaternions}.
	We argue that the representation space is very close to the structure of $\SE(3)$, yet similar enough to ordinary vector spaces to allow some of the same linear-algebraic notions that are so useful to the applied mathematician to be brought to bear on applied problems.
	The representation space is very close to the structure of $\SE(3)$ in two ways.
	First, by the fact nearly all dual quaternions are elements of $\SE(3)$ scaled by dual numbers, as shown by \Cref{thm:ProjectionSE3} of \cite{Cui2023}.
	Second, the inverse of an element of $\SE(3)$ represented on the unit dual quaternion is its dual quaternion conjugate, as we saw \Cref{sec:SE3RepresentationOnDualQuaternions}. In this way, the involutive algebra structure of the dual quaternions is intimately tied with the group structure of $\SE(3)$.
	Most importantly, despite the non-commutativity of this algebra, it is sufficiently simple to allow one to establish a spectral theorem as we saw in \Cref{sec:DQSpectralTheorem} and also the derivation of a numerical schemes to approximate eigenvectors, like we surveyed in \Cref{sec:PowerIteration}.
	We suggest that other numerical schemes, like the Arnoldi iteration and other Krylov subspace methods, may also be generalizable to the dual quaternions.
	
	Taken together, these properties indicate that the dual quaternion algebra balances well the difficulties incurred by working over non-commutative algebra with the advantages it offers.

	\section{Alignment of Two Dual Quaternion Vectors}
	\label{apx:Alignment}
	
	Let $\mathbf{g} = \left(g_{1}, \dots, g_{n}\right)^{\top} \in \SE(3)^{n}$ and $\widehat{\mathbf{g}} = \left(\widehat{g}_{1}, \dots, \widehat{g}_{n} \right)^{\top} \in \SE(3)^{n}$.
	Assume that elements of $\SE(3)$ are represented as 
	$g_{j} = \left(q_{j}, \mathbf{t}_{j}\right)  $ and $\widehat{g}_{j} = \left(\widehat{q}_{j}, \widehat{\mathbf{t}}_{j}\right)$,
	where $q_{j}$ and $\widehat{q}_{j}$ are unit quaternions with a non-negative first coordinate, representing a rotation as we discuss in \Cref{sec:SE3RepresentationOnDualQuaternions}, and $\mathbf{t}_{j}$ and $\widehat{\mathbf{t}}_{j}$ are translations represented as vectors in $\F^{3}$.
	Recall \Cref{prop:QuaternionDoubleCoverSO3} and the double cover $\varphi$ defined there.
	Aligning $\mathbf{g}$ and $\widehat{\mathbf{g}}$ amounts to finding $g = \left(q, \mathbf{t}\right)$ solving
	\begin{equation}
		\label{eq:AlignmentOptimizationProblem}
		\min_{g \in \SE(3)} \left\{ \sum_{j=1}^{n} \Ltwonorm{\widehat{\mathbf{q}_{j}} \mathbf{q} - \mathbf{q}_{j} }_{F}^{2}
		+ \sum_{j=1}^{n} \Ltwonorm{\varphi (\widehat{q}_{j}) (\mathbf{t})  + \widehat{\mathbf{t}}_{j} - \mathbf{t}_{j} }_{2}^{2} \right\}.
	\end{equation}
	Here, $\widehat{\mathbf{q}}_{j}$, $\mathbf{q}_{j}$ and $\mathbf{q}$ are real $3 \times 3$ orthogonal matrices with determinant of $1$ representing the same elements of $\SO(3)$ as $\widehat{q}_{j}$, $q_{j}$ and $q$.
	These matrices are constructed by representing $\varphi \left(q\right)$ in the standard basis of $\F^{3}$.
	Also, $\Ltwonorm{\cdot}_{2}$ is the Euclidean norm and $\Ltwonorm{\cdot }_{F}$ is the Frobenius norm.
	Solving \eqref{eq:AlignmentOptimizationProblem} amounts to finding $g$ minimizing the sum of squared distances between corresponding rotations of $\mathbf{g}$ and $\widehat{\mathbf{g}}$ and a sum of squared distances between the translations of $\mathbf{g}$ and $\widehat{\mathbf{g}}$.
	
	The two sums in \eqref{eq:AlignmentOptimizationProblem} are two decoupled optimization problems, since the left one depends only on $q$ and the right one only on $\mathbf{t}$.
	We consider these two decoupled problems separately.
	We begin with the translation sum. 
	Let $F (\mathbf{t})$ be the right sum in \eqref{eq:AlignmentOptimizationProblem}.
	We wish to find the minimum of $F$ in $\F^{3}$.
	Because the Euclidean norm is preserved under orthogonal transformations and $\varphi$ is a homomorphism, $F$ can be written as
	\begin{equation*}
		F (\mathbf{t})
		= \sum_{j=1}^{n} \Ltwonorm{ \mathbf{t} + \varphi\left(\widehat{q}_{j}^{*}\right) \left(\widehat{\mathbf{t}}_{j} - \mathbf{t}_{j}\right)}^{2}.
	\end{equation*}
	Fermat's theorem then yields that its exterema is
	\begin{equation*}
		\mathbf{t}
		= \frac{1}{n} \sum_{j=1}^{n} \varphi \left(\widehat{q}_{j}^{*}\right) \left( \mathbf{t}_{j} - \widehat{\mathbf{t}}_{j}\right),
	\end{equation*}
	exactly as in \eqref{eq:BestAlignerTranslation}.
	That it is a minimum follows from the fact $F$ is convex and defined on $\F^{3}$, and therefore can have no local maxima.
	
	The rotation sum is slightly more involved.
	First, we note that for any$\mathbf{R}_{1}$ and $\mathbf{R}_{2}$  two real $3 \times 3$ orthogonal matrices with determinant of $1$ we have
	\begin{equation*}
		\Ltwonorm{\mathbf{R}_{1} - \mathbf{R}_{2}}_{F}^{2}
		= \Ltwonorm{\mathbf{R}_{1}}_{F}^{2} 
		+ \Ltwonorm{\mathbf{R}_{2}}_{F}^{2} 
		- 2 \tr \left(\mathbf{R}_{1} \mathbf{R}_{2}^{\top}\right)
		= 3 - 2 \tr \left( \mathbf{R}_{1}^{\top} \mathbf{R}_{2} \right).
	\end{equation*}
	A minimizer of the left sum in \eqref{eq:AlignmentOptimizationProblem} is therefore a maximizer over the unit quaternions with non-negative first coordinate of 
	\begin{equation*}
		G_{1} (\mathbf{q})
		= \sum_{j=1}^{n} \tr \left( \mathbf{q}^{\top} \widehat{\mathbf{q}}_{j}^{\top} \mathbf{q}_{j} \right).
	\end{equation*}
	Finally, let $q_{1}$ and $q_{2}$ are quaternions representing the same elements of $\SO(3)$ as  $\mathbf{R}_{1}$ and $\mathbf{R}_{2}$, respectively. As we argue below, the following identity holds:
	\begin{equation}
		\label{eq:TraceIdentity}
		\tr \left( \mathbf{R}_{1}^{\top} \mathbf{R}_{2} \right)
		= 4 \Eucprod{q_{1}}{q_{2}}^{2} - 1,
	\end{equation}
	where $\Eucprod{\cdot}{\cdot}$ is the Euclidean inner product on $\F^{4}$ applied to quaternions via their obvious identification with $\F^{4}$.
	Therefore, the maximizer of  $G_{1}$ over the unit quaternions with non-negative first coordinate is the maximizer over the same domain of
	\begin{equation*}
		G (q)
		= \sum_{j=1}^{n} \Eucprod{\widehat{q}_{j} q}{q_{j}}.
	\end{equation*}
	We note that 
	\begin{equation*}
		\Eucprod{q_{1}}{q_{2}}
		= \frac{1}{2} \left( q_{1}^{*} q_{2} + q_{2}^{*} q_{1} \right),
	\end{equation*}
	and so 
	\begin{equation*}
		\Eucprod{q_{1} q}{q_{2}}
		= \frac{1}{2} \left( q^{*} q_{1}^{*} q_{2} + q_{2}^{*} q_{1} q \right)
		= \frac{1}{2} \left( q^{*} \left( q_{1}^{*} q_{2} \right) + \left( q_{1}^{*} q_{2} \right)^{*} q \right)
		= \Eucprod{q}{q_{1}^{*} q_{2}}.
	\end{equation*}
	Therefore, for any unit quaternion with non-negative first coordinate $q$ we have:
	\begin{equation*}
		G (q)
		= \sum_{j=1}^{n} \Eucprod{q}{\widehat{q}_{j}^{*} q_{j}}
		= \Eucprod{q}{ \sum_{j=1}^{n} \widehat{q}_{j}^{*} q_{j}}
		\le \Ltwonorm{q}_{2} \Ltwonorm{\sum_{j=1}^{n} \widehat{q}_{j}^{*} q_{j}}_{2}
		= \Ltwonorm{\sum_{j=1}^{n} \widehat{q}_{j}^{*} q_{j}}_{2}.
	\end{equation*}
	In the above we used the Cauchy-Schwartz inequality applied to the quaternions, again identified as elements of $\F^{4}$.
	Simple substitution into $G$ shows that the $q$ defined in \eqref{eq:BestAlignerRotation} is a unit quaternion with non-negative first coordinate which achieves this bound.
	
	It remains to prove \eqref{eq:TraceIdentity}.
	It follows from the following lemma:
	\begin{lemma}
		\label{lem:TraceIdentity}
		Let $q$ be a unit quaternion.
		Let $\mathbf{R}$ be the real $3 \times 3$ orthogonal matrix with determinant of $1$ formed by representing $\varphi(q)$ using the standard basis of $\F^{3}$.
		If $q = a + b \ii + c \jj + d \kk$, then 
		\begin{equation*}
			\tr \left(\mathbf{R}\right)
			= 4 a^{2} - 1.
		\end{equation*}
	\end{lemma}
	
	We obtain \eqref{eq:TraceIdentity} as a corollary of \Cref{lem:TraceIdentity} due to a combination of two facts. First, $\mathbf{R}_{1}^{\top} \mathbf{R}$ in \eqref{eq:TraceIdentity} expressed in quaternions is $q_{1}^{*} q_{2} $. Second, as can be worked out from \eqref{eq:QuaternionGeneratorsRelations}, 	the first coordinate of the product $q_{1}^{*} q_{2}$ is $\Eucprod{q_{1}}{q_{2}}$.
	
	\begin{proof}
		Using the axis-angle representation of $\mathbf{R}$, we consider it a rotation around axis $\mathbf{r} \in S^{2}$ by angle $\theta$.
		From \cite[Appendix A]{Huynh2009}, $\tr\left(\mathbf{R}\right) = 2 \cos \theta + 1$, while \cite[eq. (6)]{Huynh2009} yields $a = \cos \frac{\theta}{2}$.
		Therefore, using the trigonometric identity $\cos \theta = 2 \cos^{2} \frac{\theta}{2} - 1$, we obtain
		\begin{equation*}
			4 a^{2} - 1
			= 4 \cos^{2} \frac{\theta}{2} - 1
			= 2 \left( 2 \cos^{2} \frac{\theta}{2} - 1\right) + 1
			= 2 \cos \theta + 1,
		\end{equation*}
		as required.
	\end{proof}
	
	\begin{remark}
		\label{rem:RotationDistanceMetric}
		From the identities used in the proof of \Cref{lem:TraceIdentity}, it follows that the metric we defined on the unit quaternions in in \Cref{sec:SyntheticEXperimentsSetup} satisfies
		$d (q_{1}, q_{2}) 
		= 2 \arccos \left( \cos \theta \right) = 2 \theta_{1 2} $,
		where $\theta_{1 2}\in [0, \pi )$ and $\theta $ is the rotation angle in the angle-axis representation of the rotation $q_{1}^{*} q_{2}$.
	\end{remark}

	\section*{Data availability}
	The data underlying this article are available in the GitHub repository \texttt{dqsync-python} at \url{https://github.com/idohadi/dqsync-python/}.
	
	\section*{Funding}
	This work was supported by the Binational Science Foundation [2019752]; Binational Science Foundation [2020159] to T.B.; the Israeli Science Foundation [1924/21] to T.B.;  and the Deutsche Forschungsgemeinschaft [514588180] to N.S.

	\printbibliography
	
	\end{document}